\documentclass[a4paper,12pt]{amsart} 
\usepackage{amssymb, amsmath, amsthm, color, amstext, tikz, appendix, mathrsfs, mathabx,stmaryrd}
\usepackage{dsfont}
\usepackage[colorlinks=true, breaklinks=true, linkcolor=black, citecolor=blue, urlcolor=red]{hyperref} 
\usepackage[english]{babel}
\usepackage [all,cmtip]{xy}
\usepackage{epigraph}

\numberwithin{equation}{section}
\newtheorem{letterthm}{Theorem}

\newtheorem{lettercor}[letterthm]{Corollary}

\newtheorem{letterprob}[letterthm]{Problem}

\newtheorem{theorem}{Theorem}[section]

\newtheorem{lemma}[theorem]{Lemma}

\newtheorem{proposition}[theorem]{Proposition}

\newtheorem{definition}[theorem]{Definition}
\newtheorem{notation}[theorem]{Notation}
\newtheorem{remark}[theorem]{Remark}

\newcommand{\act}{\curvearrowright}
\DeclareMathOperator{\ad}{ad}
\newcommand{\al}{\alpha}
\DeclareMathOperator{\Aut}{Aut}
\newcommand{\cC}{\mathcal C}
\newcommand{\C}{\mathbf C}
\newcommand{\cD}{\mathcal D}

\newcommand{\de}{\delta}

\newcommand{\cF}{\mathcal F}

\newcommand{\cAF}{\mathcal{AF}}
\newcommand{\cG}{\mathcal G}
\newcommand{\scrG}{\mathscr G}
\DeclareMathOperator{\Gr}{Gr}
\newcommand{\fH}{\mathfrak H}
\newcommand{\scrH}{\mathscr H}
\DeclareMathOperator{\Hilb}{Hilb}
\newcommand{\ga}{\gamma}
\newcommand{\Ga}{\Gamma}
\DeclareMathOperator{\id}{id}

\newcommand{\fK}{\mathfrak K}
\newcommand{\fL}{\mathfrak L}
\newcommand{\scrL}{\mathscr L}
\newcommand{\La}{\Lambda}
\newcommand{\la}{\lambda}
\newcommand{\N}{\mathbf{N}}
\DeclareMathOperator{\ob}{ob}
\newcommand{\ot}{\otimes}
\newcommand{\ov}{\overline}
\DeclareMathOperator{\Path}{Path}
\newcommand{\Q}{\mathbf{Q}}
\newcommand{\R}{\mathbf{R}}
\DeclareMathOperator{\shift}{shift}
\DeclareMathOperator{\St}{State}
\newcommand{\cSF}{\mathcal{SF}}
\DeclareMathOperator{\source}{source}
\newcommand{\fT}{\mathfrak T}
\DeclareMathOperator{\target}{target}
\newcommand{\ti}{\tilde}
\DeclareMathOperator{\Tens}{Tens}
\newcommand{\cU}{\mathcal U}
\newcommand{\cV}{\mathcal V}
\newcommand{\varep}{\varepsilon}
\newcommand{\scrX}{\mathscr X}
\newcommand{\Z}{\mathbf{Z}}

\begin{document}

\title[Haagerup property for wreath products]{Haagerup property for wreath products constructed with Thompson's groups}
\thanks{
AB is supported by the Australian Research Council Grant DP200100067 and a University of New South Wales Sydney Starting Grant.}
\author{Arnaud Brothier}
\address{Arnaud Brothier\\ School of Mathematics and Statistics, University of New South Wales, Sydney NSW 2052, Australia}
\email{arnaud.brothier@gmail.com\endgraf
\url{https://sites.google.com/site/arnaudbrothier/}}
\maketitle

\begin{abstract}
Using recent techniques introduced by Jones we prove that a large family of discrete groups and groupoids have the Haagerup property. 
In particular, we show that if $\Ga$ is a discrete group with the Haagerup property, then the permutational restricted wreath product $\oplus_{\Q_2}\Ga\rtimes V$ obtained from the group $\Ga$ and the usual action of Richard Thompson's group $V$ on the dyadic rational $\Q_2$ of the unit interval has the Haagerup property.
\end{abstract}

\hspace{14cm} \textit{\`A C\'ecile} 


\section{Introduction}
In the 1930s Ore gave necessary and sufficient conditions for a semi-group to embed in a group, see \cite{Maltsev53}.
Similar properties can be defined for categories giving a calculus of fractions and providing the construction of a groupoid (of fractions) and in particular groups, see \cite{GabrielZisman67}.
Richard Thompson's groups $F\subset T\subset V$ arise in that way by considering certain diagrammatic categories of forests, see \cite{Brown87,Cannon-Floyd-Parry96} and \cite{Belk04,Jones16-Thompson} for the categorical framework.

Recently, Jones discovered a very general process that constructs a group action (called {\it Jones' action}) $\pi_\Phi:G_\cC\act X_\Phi$ from a functor $\Phi:\cC\to\cD$ where $\cC$ is a category admitting a calculus of fractions and where $G_\cC$ is the group of fractions associated to $\cC$ (and a fixed object) \cite{Jones17-Thompson,Jones16-Thompson}, see also the survey \cite{Brothier20-survey}.
The action remembers some of the structure of the category $\cD$ and, in particular, if the target category is the category of Hilbert spaces (with linear isometries for morphisms), then $\pi_\Phi$ is a unitary representation (in that case we call it a {\it Jones' representation}).
This provides large families of unitary representations of the Thompson's groups \cite{Brothier-Jones19,Brothier-Jones19bis,ABC19,Jones21,Brothier-Wijesena22}.
Certain coefficients of Jones' representations can be explicitly computed via algorithms which makes them very useful for understanding analytical properties of groups of fractions.
This article uses for the first time Jones' machinery for proving that new classes of groups (and groupoids) satisfy the Haagerup property.

{\bf Haagerup property.}
Recall that a \textit{discrete} group has the Haagerup property if it admits a net of positive definite functions vanishing at infinity and converging pointwise to one \cite{Akemann-Walter81}, see also the book \cite{CCJJV01} and the recent survey \cite{Valette17}.
It is a fundamental property having applications in various fields such as group theory, ergodic theory, operator algebras, and K-theory for instance.
The Haagerup property is equivalent to Gromov's a-(T)-meanability (i.e.~the group admits a proper affine isometric action on a Hilbert space) and, as suggested by Gromov's terminology, it is a strong negation of Kazhdan's Property (T): a discrete group having both properties is necessarily finite \cite{Gromov93}.
One additional motivation to study the Haagerup property is given by a deep theorem of Higson and Kasparov: a group having the Haagerup property satisfies the Baum-Connes conjecture (with coefficients) and in particular satisfies the Novikov conjecture \cite{HigsonKasparov01}.

{\bf Wreath products.}
The class of groups with the Haagerup property contains amenable groups and many other since it is closed under taking free products and even graph products \cite{AntolinDreesen13}.
However, it is not closed under taking extensions and in particular under taking wreath products. 
We call {\it wreath product} (instead of permutational restricted wreath product) a group of the form $\Ga\wr_{X}\La:=\oplus_X\Ga\rtimes \La$ where $\Ga,\La$ are groups, $X$ is a $\La$-set, $\oplus_X\Ga$ is the group of \textit{finitely supported} maps from $X$ to $\Ga$, and the action $\La\act \oplus_X\Ga$ consists in shifting indices using the $\La$-set structure of $X$.
It is notoriously a difficult problem to prove that a wreath product has the Haagerup property or not.
Cornulier, Stalder and Valette showed that, if $\Ga$ and $\La$ are discrete groups with the Haagerup property, then so does the wreath product $\oplus_{g\in \La} \Ga\rtimes\La$ and so does $\oplus_{g\in \La/\Delta} \Ga\rtimes \La$ where $\Delta$ is a normal subgroup of $\La$ satisfying that the quotient group $\La/\Delta$ has the Haagerup property \cite{CornulierStalderValette12}.
See also \cite{Cornulier18} where the later result was extended to commensurated subgroups $\Delta< \La.$
However, no general criteria exists for wreath products like $\oplus_{X} \Ga\rtimes \La$ where $X$ is \textit{any} $\La$-set. 
Moreover, there exist many examples of wreath products $\oplus_{X} \Ga\rtimes \La$ having relative Kazhdan's property (T) thus not having the Haagerup property even when $\Ga,\La$ have it, see \cite{CornulierStalderValette12}.

{\bf Thompson groups.}
There have been increasing results on analytical properties of Thompson's groups $F\subset T\subset V$: Reznikoff showed that Thompson's group $T$ does not have Kazhdan's Property (T) and Farley proved that $V$ has the Haagerup property \cite{Reznikoff01,Farley03-H}.
Independently, the works of Ghys-Sergiescu and Navas on diffeomorphisms of the circle implies that $F$ and $T$ do not have Kazhdan's Property (T) \cite{Ghys-Sergiescu87,Navas02}.
Using Jones' technology, Jones and the author constructed explicit positive definite maps on $V$. 
This permitted to give two independent short arguments proving that $V$ does not hat Kazhdan's Property (T) and that $T$ has the Haagerup property \cite{Brothier-Jones19}.

{\bf Wreath products using Thompson's groups.}
In this article we consider wreath products built from actions of Thompson's groups.
More precisely, let $\Q_2$ be the set of dyadic rationals in $[0,1)$ and consider the usual action $V\act \Q_2$. 
Given any group $\Ga$ we may form the wreath product 
$$\Ga\wr_{\Q_2}V:=\oplus_{\Q_2}\Ga\rtimes V.$$
More generally, if $\theta$ is an automorphism of $\Ga$ we may form the {\it twisted} wreath product 
$$\Ga\wr_{\Q_2}^\theta V$$
where the action $V\act \oplus_{\Q_2}\Ga$ is given by the formula:
$$(v\cdot a)(x) = \theta^{\log_2(v'(v^{-1}x))}(a(v^{-1}x)) \text{ for all } v\in V, a\in \oplus_{\Q_2}\Ga, x\in \Q_2.$$

Using Jones' technology we define in this article a net of coefficients vanishing at infinity on the larger group $V$ and thus reproving Farley's result.
By mixing these coefficients together with representations of a given group $\Ga$ (see below for details) we manage to prove the following result.

\begin{letterthm}\label{THA}
Consider a discrete group $\Ga$ and an automorphism of it $\theta\in\Aut(\Ga)$.
If $\Ga$ has the Haagerup property, then so does the twisted wreath product $\Ga\wr_{\Q_2}^\theta V.$
\end{letterthm}

{\bf New examples.}
Wreath products obtained in Theorem \ref{THA} were not previously known to have the Haagerup property. 
Moreover, we provide the first analytic but not geometric proof showing that a wreath product has the Haagerup property. Indeed, previous techniques were based on showing that the group admits a proper isometric action (for example using an action on a space with walls). We thank Adam Skalski for pointing this out.

Note that if $\Ga$ is finitely presented, then so does the wreath product by a result of Cornulier \cite{Cornulier06}.
Further, if $\Ga$ satisfies the homological (resp.~topological) finiteness property of being of type $FP_m$ (resp.~$F_m$) for any $m\geq 1$ or $m=\infty$, then so does the wreath product $\Ga\wr_{\Q_2}V$ by Bartholdi, Cornulier, and Kochloukova \cite{Bartholdi-Cornulier-Kochloukova15}, see also \cite[Section 4.3]{Brothier22-HPM}.
We obtain the first examples of finitely presented wreath products (or of any type $F_m$ or $FP_m$ with $m\geq 2$) that have the Haagerup property for a nontrivial reason that is:
the group acting (here $V$) is nonamenable and the base space (here $\Q_2$) is not finite.
We are grateful to Yves de Cornulier for making this observation.

{\bf Pairwise non-isomorphic examples.}
Since the class of groups satisfying the Haagerup property is closed under taking subgroups we obtain the same statement in Theorem \ref{THA} when we replace $V$ by the smaller Thompson's groups $F$ and $T$.
Moreover, note that we obtain infinitely many pairwise non-isomorphic new examples. 
Indeed, we previously proved that if $\Ga\wr_{\Q_2}^\theta V$ is isomorphic to $\ti\Ga\wr_{\Q_2}^{\ti\theta} V,$ then there exists an isomorphism $\beta:\Ga\to\ti\Ga$ and $h\in \ti\Ga$ satisfying $\ti\theta= \ad(h)\circ \beta\theta\beta^{-1}$, see \cite[Theorem 4.12]{Brothier22}.
The same conclusion holds when $V$ is replaced by $F$ or $T$.

We were able to prove Theorem \ref{THA} because $\Ga\wr_{\Q_2}^\theta V$ is the fraction group of a certain category to which we can apply efficiently Jones' technology.
These specific groups previously appeared independently in two other frameworks. 
Indeed, Tanushevski considered those as well as Witzel and Zaremsky \cite{Tanushevski16,Witzel-Zaremsky18}.
Note that the approach of Witzel and Zaremsky, known as {\it cloning systems}, is a systematisation of a construction due to Brin of the so-called braided Thompson group \cite{Brin07}. 
We refer the reader to the appendix of \cite{Brothier21} for an extensive discussion on these three independent constructions.

A similar diagrammatic construction provides the following groups 
$$C(\mathfrak C,\Ga)\rtimes V$$
where $\mathfrak C:=\{0,1\}^\N$ is the Cantor space and $C(\mathfrak C,\Ga)$ the group of all continuous maps from $\mathfrak C$ to $\Ga$ (i.e.~the locally constant maps) equipped with the pointwise multiplication.
The action $V\act C(\mathfrak C,\Ga)$ is the one induced by the classical action $V\act\mathfrak C$ on the Cantor space.
Even if these groups arise similarly from categories than the wreath products of Theorem \ref{THA} we have been unable to understand their analytic properties leading to the following problem.

\begin{letterprob}\label{prob}
Assume that $\Ga$ is a discrete group with the Haagerup property. 
Is is true that $C(\mathfrak C,\Ga)\rtimes V$ has the Haagerup property?
\end{letterprob}

We refer the reader to \cite{Brothier21} where we extensively study this specific class of groups. 

{\bf Proof of the main result.}
The proof is made in three steps. 
Step one: we construct a family of functors starting from the category of binary symmetric forests (the category for which Thompson's group $V$ is the group of fractions) to the category of Hilbert spaces giving us a net of positive definite coefficients on $V$.
We prove that this net is an approximation of the identity satisfying the hypothesis of the Haagerup property and thus reproving Farley's result that $V$ has the Haagerup property \cite{Farley03-H}.

Step two: given any group $\Ga$ we construct a category with a calculus of left-fractions whose group of fractions is isomorphic to the wreath product $\oplus_{\Q_2} \Ga\rtimes V.$ 
Elements of $V$ are described by (equivalence classes) of triples $(t,\pi,t')$ where $t,t$ are trees with same number of leaves and $\pi$ a bijection between the leaves of $t$ and leaves of $t'$.
For the larger group $\oplus_{\Q_2} \Ga\rtimes V$ we have a similar description with an extra data being a labeling of the leaves of $t,t'$ by elements of the group $\Ga.$ 

Step three: given a unitary representation of $\Ga$ and a functor of step one we construct a functor starting from the larger category constructed in step two and ending in Hilbert spaces. This provides a net of coefficients for the wreath product indexed by representations of $\Ga$ and functors of step one. We then extract from those coefficients a net satisfying the assumptions of the Haagerup property.

Step two is not technically difficult but resides on the following key observation:
given \textit{any} functor $\Xi:\cF\to\Gr$ from the category of forests to the category of groups we obtain, using Jones' machinery, an action $\al_\Xi:F\act \scrG_\Xi$ of Thompson's group $F$ on a certain limit group $\scrG_\Xi$. 
In certain cases (for example when $\Xi$ is monoidal) we can extend $\al_\Xi$ into a $V$-action.
We observe that there exists a category $\cC_\Xi$ whose group of fractions is isomorphic to the semi-direct product $\scrG_\Xi\rtimes_{\al_\Xi}V$ and this observation works more generally whatever the initial category is, see Remark \ref{rem:LargeCat}.
Moreover, the category $\cC_\Xi$ and its group of fractions have very explicit forest-like descriptions allowing us to extend techniques built to study Thompson's group $V$ to the larger group of fractions of $\cC_\Xi$.
By choosing wisely the functor $\Xi$ we obtain that the group of fractions of $\cC_\Xi$ is isomorphic to $\oplus_{\Q_2} \Ga\rtimes V.$ 
This procedure shows that certain semi-direct products $\scrG\rtimes V$ (or more generally $\scrG\rtimes G_\cD$ where $G_\cD$ is a group of fractions) have a similar structure than $V$ (resp.~$G_\cD$) and thus we might hope that certain properties of $V$ (resp.~$G_\cD$) that are not necessarily closed under taking extension might still be satisfied by $\scrG\rtimes V$ (resp.~$\scrG\rtimes G_\cD$).
Note that the groups appearing in Problem \ref{prob} arise in that way.

The main technical difficulty of the proof of Theorem \ref{THA} resides in steps one and three; in particular in showing that the coefficients are vanishing at infinity. 
In step one, we define functors $\Phi:\cF\to\Hilb$ from binary forests to Hilbert spaces such that the image $\Phi(t)$ of a tree $t$ with $n+1$ leaves is a sum of $2^n$ operators. 
We let this operator acting on a vector obtaining  a sum of $2^n$ vectors. 
To this functor we associate a coefficient for Thompson's group $V$ where a group element described by a fraction of symmetric trees with $n+1$ leaves is sent to $2^n\times 2^n$ inner products of vectors. 
We show that if the fraction is irreducible, then most of those inner products are equal to zero implying that the coefficient vanishes at infinity.
In step three we adapt this strategy to a larger category where leaves of trees are decorated with element of the group $\Ga$ that requires the introduction of more sophisticated functors. 
This extension of step one is not straightforward.
One of the main difficulty comes from the fact that fractions of decorated trees are harder to reduce. 
For example, there exists a sequence of tree $t_n$ with $n$ leaves such that $\dfrac{g_nt_n}{t_n}$ is a reduced fraction where $g_n$ has only one nontrivial entry equal to a fix $x\in\Ga$ (see Section \ref{sec:largercat} for notations).
If we forget $g_n$, then the fraction $\dfrac{t_n}{t_n}$ corresponds to the trivial element of Thompson's group $F$.
Therefore, a naive construction of a functor that would treat independently data of trees and elements of $\Ga$ cannot produce coefficients that vanishes at infinity since it will send $\dfrac{g_nt_n}{t_n}$ to a nonzero quantity depending only on $x$.

The argument works identically for \textit{countable} and \textit{uncountable} discrete groups $\Ga$.
Interestingly, the coefficients of Thompson's group $V$ appearing in step one are not the one constructed by Farley nor the one previously constructed by the author and Jones but coincide when we restrict those coefficients to the smaller Thompson's group $T$, see Remark \ref{rem:Farley} and the original articles \cite{Farley03-H,Brothier-Jones19}.

We could have given a single proof showing that if $\Ga$ has the Haagerup property, then so is the associated (possibly twisted) wreath product $\Ga\wr_{\Q_2}V.$
Although, for pedagogical reasons we choose to provide several proofs for various groups. 
This permits to understand easily the scheme of the proof and to appreciate the gap of difficulties between various cases. We thus prove the Haagerup property for $F$, then for $T$, then for $V$, then for $\Ga\wr_{\Q_2}V$, and finally for a twisted version of it. The largest gaps of technicality resides between $T$ and $V$ and between $V$ and the wreath product.

The proof of Theorem \ref{THA} is based on a categorical and functorial approach that is more natural to use for studying \textit{groupoids}. 
We present such a groupoid approach allowing now {\it $k$-ary forests} rather than only {\it binary trees}.
This leads to the following theorem:

\begin{letterthm}\label{THB}
Consider a triple $(\Ga,\theta, k)$ where $\Ga$ is a group, $\theta:\Ga\to\Ga$ an injective morphism, and $k\geq 2.$
There exists a unique monoidal category $\cC$ (see Section \ref{sec:largercat}) whose objects are the natural numbers and morphisms from $n$ to $m$ are $k$-ary forests with $n$ roots, $m$ leaves together with a permutation of the leaves and a labelling of the leaves with elements of $\Ga$. 
Moreover, the composition of morphisms satisfies the relation 
$$Y_k\circ g = (\theta(g),e,\cdots,e)\circ Y_k$$ where $g\in \Ga$ and $Y_k$ is the unique $k$-ary tree with $k$ leaves. 

If $\cG_\cC$ is the universal groupoid of $\cC$ and $\Ga$ is a discrete group that has the Haagerup property, then $\cG_\cC$ has the Haagerup property.
\end{letterthm}

Note that the groups appearing in Problem \ref{prob} corresponds to the category built from the relation $Y\circ g=(g,g)\circ Y$ for $g\in\Ga.$

If $\cG_{\cSF_k}$ is the universal groupoid of the category of $k$-ary symmetric forests, then the automorphism group (i.e.~the isotropy group) $\cG_{\cSF_k}(r,r)$ of the object $r$ is isomorphic to the Higman-Thompson group $V_{k,r}$, see \cite{Higman74,Brown87}.
Further, by adding decoration of the leaves with a group $\Ga$ and setting $\theta=\id_\Ga$ the identity, we obtain that the isotropy group at the object $r$ is isomorphic to the wreath product 
$$\Ga\wr_{Q_k(0,r)}V_{k,r}=\oplus_{\Q_k(0,r)} \Ga\rtimes V_{k,r}$$ 
where $V_{k,r}\act \Q_k(0,r)$ is the usual action of Higman-Thompson's group $V_{k,r}$ on the set of $k$-adic rationals inside $[0,r)$.
If $\theta$ is a nontrivial automorphism, then we obtain a twisted wreath product similarly than in the binary case.

\begin{lettercor}\label{COR}
Let $\Ga$ be a discrete group with the Haagerup property and $\theta\in\Aut(\Ga)$ an automorphism.
Denote by $\Ga\wr_{\Q_k(0,r)}^\theta V_{k,r}$ the twisted wreath product associated to the usual action $V_{k,r}\act \Q_k(0,r)$ and $\theta$ for $k\geq 2,r\geq 1.$
We have that $\Ga\wr_{\Q_k(0,r)}^\theta V_{k,r}$ has the Haagerup property.
\end{lettercor}

This corollary generalises Theorem \ref{THA} which corresponds to the case $k=2$ and $r=1$.

Apart from the introduction this article contains five other sections and a short appendix.
In Section \ref{sec:prelim} we introduce all necessary background concerning Thompson's groups, groups of fractions and Jones' actions. 
We then explain how to build larger categories from functors and how their group of fractions are isomorphic to certain wreath products.
In Section \ref{sec:FT} we provide short and simple proofs that $F$ and $T$ have the Haagerup property by constructing an explicit net of linear isometries and by considering associated positive definite maps. We then easily observe that they vanish at infinity and converge pointwise to 1.
In Section \ref{sec:VH}, we prove that Thompson's group $V$ has the Haagerup property by refining substantially the proofs for $F$ and $T$ but by keeping the same strategy. It is still easy to see that the positive definite maps converge pointwise to 1. Although, it is much harder to show that they vanish at infinity.
In Section \ref{sec:wreathprod}, we prove Theorem \ref{THA}.
We explain how to build matrix coefficient on larger fraction groups. We then follow a similar but more technical strategy.
In Section \ref{sec:groupoidap}, we adopt a groupoid approach. We introduce all necessary definitions and constructions that are easy adaptations of the group case. We then prove Theorem \ref{THB} and deduce Corollary \ref{COR}.
In a short appendix we provide a different description of Jones' actions using a more categorical language.

\subsection*{Acknowledgement}
We warmly thank Sergei Ivanov, Richard Garner and Steve Lack for enlightening discussions concerning category theory. We thank Adam Skalski for making key comments to us regarding the results and techniques used in this article.
We are grateful to Yves de Cornulier and Vaughan Jones for very constructive comments on an earlier version of this manuscript and to Dietmar Bisch, Matt Brin and Yash Lodha for their enthusiasm and encouragements.
Finally, we thank Christian de Nicola Larsen for pointing out some typos and technical subtelties in an earlier version of the manuscript.

\section{Preliminaries}\label{sec:prelim}

\subsection{Groups of fractions}
We say that a category $\cC$ is {\it small} if its collections of objects and morphisms are both sets. The collection of morphisms of $\cC$ from $a$ to $b$ is denoted by $\cC(a,b)$. If $f\in\cC(a,b)$, then we say that $a$ is the source and $b$ the target of $f$. As usual we compose from right to left, thus the source of $g\circ f$ is the source of $f$ and its target the target of $g.$
When we write $g\circ f$ we implicitly assume that $g$ is composable with $f$ meaning that the target of $f$ is equal to the source of $g.$
We sometime write $gf$ for $g\circ f.$

\subsubsection{General case}
We explain how to construct a group from a small category together with the choice of one of its object.
We refer to \cite{Jones16-Thompson} for details on this specific construction and to \cite{GabrielZisman67} for the general theory of calculus of fractions.

Let $\cC$ be a small category and $e$ an object of $\cC$ satisfying:
\begin{enumerate}
\item (Left-Ore's condition at $e$) If $p,q$ have same source $e$, then there exists $h,k$ such that $hp=kq.$
\item (Weak left-cancellative at $e$) If $pf=qf$ where $f$ has source $e$, then there exists $g$ such that $gp=gq.$\\
\end{enumerate}
We say that such a category admits a {\it calculus of left-fractions in $e$}.

\begin{proposition}
Let $G_\cC$ be the set of pairs $(t,s)$ of morphisms with source $e$ and common target that we quotient by the equivalence relation generated by $(t,s)\sim (ft,fs)$. Denote by $\dfrac{t}{s}$ the equivalence class of $(t,s)$ that we call a fraction.
The set of fractions admits a multiplication $\cdot$ such that 
$$\dfrac{t}{s}\cdot\dfrac{t'}{s'} =\dfrac{ft}{f's'} \text{ for any $f,f'$ satisfying } fs = f't' .$$
This confers a group structure to $G_\cC$ such that $\dfrac{s}{t}$ is the inverse of $\dfrac{t}{s}$ and thus $\dfrac{t}{t}$ is the identity for all $t$.
We call $G_\cC$ the {\it group of fractions} of $(\cC,e)$ or of $\cC$ if the context is clear. 
\end{proposition}

\begin{proof}
Given two pairs $(t,s),(t',s')$ as above there exists by Ore's condition at $e$ some morphisms $f,f'$ satisfying $fs=f't'.$
We write $(t,s)_{f,f'}(t',s')$ for the product giving $(ft,f's').$
We claim that $\dfrac{ft}{f't'}$ only depends on the classes $\dfrac{t}{s}$ and $\dfrac{t'}{s'}.$
Consider another pair of morphisms $g,g'$ satisfying $gs=g't'$ and observe that $(t,s)_{g,g'}(t',s')=(gt,g's').$
By Ore's condition at $e$ there exists $h,k$ such that $hfs=kgs.$
Observe that $$hf't' = hfs = kgs = kg't'.$$
By the weak cancellation property at $e$ there exists $b$ such that $bhf'=bkg'$.
Moreover, since $hfs=kgs$ we have $bhfs=bkgs$ and thus by the weak cancellation property at $e$ there exists $a$ such that $abhf=abkg.$
We obtain the equalities:
\begin{enumerate}
\item $bhf' = bkg'$;
\item $abhf=abkg$.
\end{enumerate}
Observe that 
\begin{align*}
\dfrac{ft}{f's'} & = \dfrac{bhft}{bhf's'} = \dfrac{bhft}{bkg's'} \text{ by } (1)\\
& = \dfrac{abhft}{abkg's'}  = \dfrac{abkgt}{abkg's'} \text{ by } (2)\\
& = \dfrac{gt}{g's'}.
\end{align*}
This proves the claim. The rest of the proposition follows easily.
\end{proof}

When $\cC$ satisfies the property of above for any of its object we say that it admits a {\it calculus of left-fractions}. 
This is then the right assumptions for considering a groupoid of fractions, see Section \ref{sec:univgroupoid}. 
We will be mostly working with categories of forests defined below and refer to \cite{Cannon-Floyd-Parry96,Belk04} for more details about this case.
Note that those categories satisfy stronger axioms as they are cancellative (right and left) and satisfies Ore's property at any object. 

\begin{remark}
We have followed the original conventions appearing in the first articles on Jones' technology. 
Unfortunately they are different from the more recent articles when we consider right-fractions instead of left-fractions.
Note that $\dfrac{t}{s}$ corresponds formally to $t^{-1}\circ s$ and is sometime denoted $[t,s]$. 
In more recent articles we often write Frac$(\cC)$ for the fraction groupoid of a category $\cC$ and Frac$(\cC,e)$ rather than $G_\cC$ for the fraction group of $\cC$ at the object $e$.\\
The formal notation permits to check easily the identities $\dfrac{t}{s}\cdot \dfrac{s}{u}=\dfrac{t}{u}$ by computing $(t^{-1}\circ s) \circ (s^{-1}\circ u)$ and check that $\dfrac{f\circ t}{f\circ s}=\dfrac{t}{s}$ by computing $(f\circ t)^{-1}\circ (f\circ s).$
\end{remark}

\subsubsection{Categories of forests and Thompson's groups}\label{sec:forest}

{\bf Trees and forests.}
Let $\cF$ be the category of finite ordered rooted binary forests whose objects are the nonzero natural numbers $\N^*:=\{1,2,\cdots\}$ and morphisms $\cF(n,m)$ the set of forests with $n$ roots and $m$ leaves. 
We represent them as diagram in the plane $\R^2$ whose roots and leaves are distinct points in $\R\times\{0\}$ and $\R\times \{1\}$ respectively and are counted from left to right starting from $1$.
For example 
\newcommand{\forest}{
\begin{tikzpicture}[baseline = .4cm]
\draw (0,0)--(0,1);
\draw (1,0)--(1,2/3);
\draw (1,2/3)--(2/3,1);
\draw (1,2/3)--(4/3,1);
\draw (2,0)--(2,1/3);
\draw (2,1/3)--(5/3,1);
\draw (2,1/3)--(7/3,1);
\draw (13/6,2/3)--(2,1);
\end{tikzpicture}
}
$$f = \ \forest \ $$
is a morphism from $3$ to $6$.
A vertex $v$ of a tree has either zero or two descendants $v_l,v_r$ that are placed on the top left and top right, respectively, of the vertex $v$. The edge joining $v$ and $v_l$ (resp.~$v_r$) is called a left-edge (resp.~a right-edge). 
We compose forests by stacking them vertically so that $f\circ q$ is the forest obtained by stacking $f$ on top of $q$  where the $i$-th {\it root} of $f$ is attached to the $i$-th {\it leaf} of $q$. 
We obtain a diagram in the strip $\R\times [0,2]$ that we rescale in $\R\times [0,1].$
For example, if 
\newcommand{\funfun}{
\begin{tikzpicture}[baseline=.2cm, scale = .4]
\draw (0,0)--(0,-.5);
\draw (0,0)--(-1,2);
\draw (-.5,1)--(0,2);
\draw (0,0)--(1,2);
\end{tikzpicture}
}
$$t = \ \funfun\ ,$$
then 
\newcommand{\compo}{
\begin{tikzpicture}[baseline = -.2cm, scale = .6]
\draw (1,-2)--(1,-2.5);
\draw (1,-2)--(0,0);
\draw (.5,-1)--(1,0);
\draw (1,-2)--(2,0);
\draw (0,0)--(0,1);
\draw (1,0)--(1,2/3);
\draw (1,2/3)--(2/3,1);
\draw (1,2/3)--(4/3,1);
\draw (2,0)--(2,1/3);
\draw (2,1/3)--(5/3,1);
\draw (2,1/3)--(7/3,1);
\draw (13/6,2/3)--(2,1);
\end{tikzpicture}
}
$$\begin{small}f\circ t = \ \compo\ .\end{small}$$

A {\it tree} is a forest with one root and conversely a forest with $n$ roots is nothing else than a list of $n$ trees.

{\bf Thompson's group $F$.}
The category $\cF$ admits a calculus of left-fractions. 
We consider the object $1$ and note that morphisms with source $1$ are trees.
The associated group of fractions $G_\cF$ is isomorphic to Thompson's group $F$.

{\bf Fraction.}
By definition, any element $g\in F$ can be expressed as a fraction $\dfrac{t}{s}$ where $t,s$ are trees with the same number of leaves say $n$.
Moreover, if $t'=f\circ t$ and $s'=f\circ s$ where $f$ is any forest having $n$ roots, then $g$ is also expressed by the fraction $\dfrac{t'}{s'}$.

{\bf Elementary forest.}
For any $1\leq i\leq n$ we consider the forest $f_{i,n}$ (denoted by $f_i$ if the context is clear) the forest with $n$ roots and $n+1$ leaves where the $i$-th tree of $f_{i,n}$ has two leaves and all other trees are trivial.
We say that $f_{i,n}$ is an {\it elementary forest}.
Here is an example:
\newcommand{\forestftwo}{
\begin{tikzpicture}[baseline=.4cm]
\draw (0,0)--(0,2/3);
\draw (0,2/3)--(-1/3,1);
\draw (0,2/3)--(1/3,1);
\draw (-2/3,0)--(-2/3,1);
\draw (2/3,0)--(2/3,1);
\draw (1,0)--(1,1);
\end{tikzpicture}
}
$$f_{2,4} = \ \forestftwo\ .$$
Note that every forest is a finite composition of elementary forests.

\begin{notation}
We write $\fT$ for the collection of all finite ordered rooted binary trees and by $Y=f_{1,1}$ the unique tree with two leaves and $I$ the unique tree with one leaf that we call the trivial tree.
By tree we always mean an element of $\fT$.
\end{notation}

{\bf Symmetric forests and Thompson's group $V$.}
Consider now the category of \textit{symmetric forests} $\cSF$ with objects $\N^*$ and morphisms 
$$\cSF(n,m)=\cF(n,m)\times S_m$$ 
where $S_m$ is the symmetric group of $m$ elements. We call an element of $\cSF(n,m)$ a {\it symmetric forest} and, if $n=1$, a {\it symmetric tree}.
Graphically we interpret a morphism $(p,\sigma)\in\cSF(n,m)$ as the concatenation of two diagrams.
On the bottom we have the diagram explained above for the forest $p$ in the strip $\R\times [0,1].$
The diagram of $\sigma$ is the union of $m$ segments 
$$[x_i , x_{\sigma(i)}+(0,1)], i=1,\cdots,m$$ 
in $\R\times [1,2]$ where the $x_i$ are $m$ distinct points in $\R\times \{1\}$ such that $x_i$ is on the left of $x_{i+1}.$
The full diagram of $(p,\sigma)$ is obtained by stacking the diagram of $\sigma$ on top of the diagram of $p$ such that $x_i$ is the $i$-th leaf of $p$.
If we consider the permutation $\tau$ such that $\tau(1)=2,\tau(2)=3, \tau(3)=1$, then its corresponding diagram is
\newcommand{\perm}{
\begin{tikzpicture}[baseline = .4cm]
\draw (0,0)--(1,1);
\draw (1,0)--(2,1);
\draw (2,0)--(0,1);
\end{tikzpicture}
}
$$\perm\ .$$
If $t=\funfun$, then the diagram associated to $(t,\tau)$ is
\newcommand{\taut}{
\begin{tikzpicture}[baseline = .4cm, scale=.5]
\draw (0,0)--(0,-.5);
\draw (0,0)--(-1,2);
\draw (-.5,1)--(0,2);
\draw (0,0)--(1,2);
\draw (-1,2)--(0,3);
\draw (0,2)--(1,3);
\draw (1,2)--(-1,3);
\end{tikzpicture}
}
$\taut \ .$

{\bf Two kinds of morphisms.}
We interpret the morphism $(p,\sigma)$ as the composition of the morphisms $(I_m,\sigma) \circ (p,\id)$ where $I_m$ is the trivial forest with $m$ roots and $m$ leaves (thus $m$ trivial trees next to each other) and $\id$ is the trivial permutation.
By identifying $\sigma$ with $(I_m,\sigma)$ and $p$ with $(p,\id)$ we obtain that $(p,\sigma)=\sigma\circ p.$
We have already defined compositions of forests in the description of the category $\cF$. The composition of permutations is the usual one. 
It remains to explain the composition of a forest with a permutation.
Consider a permutation $\tau$ of $n$ elements and a forest $p$ with $n$ roots and $m$ leaves and let $l_i$ be the number of leaves of the $i$-th tree of $p.$
We define the composition as:
$$p\circ \tau = S(p,\tau) \circ \tau(p),$$
where $\tau(p)$ is the forest obtained from $p$ by permuting its trees such that the $i$-th tree of $\tau(p)$ is the $\tau(i)$-th tree of $p$ and $S(p,\tau)$ is the permutation corresponding to the diagram obtained from $\tau$ where the $i$-th segment $[x_i , x_{\tau(i)} + (0,1)]$ is replaced by $l_{\tau(i)}$ parallel segments.
\newcommand{\permtilde}{
\begin{tikzpicture}[baseline = .4cm]
\draw (0,0)--(1,1);
\draw (1,0)--(2,1);
\draw (2,0)--(3,1);
\draw (3,0)--(4,1);
\draw (4,0)--(5,1);
\draw (5,0)--(6,1);
\draw (6,0)--(0,1);
\end{tikzpicture}
}
\newcommand{\foresttilde}{
\begin{tikzpicture}[baseline = .4cm]
\draw (2,0)--(2,1);
\draw (0,0)--(0,2/3);
\draw (0,2/3)--(-1/3,1);
\draw (0,2/3)--(1/3,1);
\draw (1,0)--(1,1/3);
\draw (1,1/3)--(2/3,1);
\draw (1,1/3)--(4/3,1);
\draw (7/6,2/3)--(1,1);
\end{tikzpicture}
}
{For example, if we consider the forest 
$$f = \forest$$ and the permutation $$\tau=\perm\ ,$$ then 
$$ f\circ \tau = S(f,\tau)\circ \tau(f)$$ 
where 
$$\tau(f) = \foresttilde$$ \text{ and } $$S(f,\tau) = \permtilde \ .$$
}
This is a category admitting a calculus of left-fractions whose group of fractions associated to $(\cSF,1)$ is isomorphic to Thompson's group $V$.
Note that the relations between forests and permutations can be interpreted as a Brin-Zappa-Sz\'ep product of the category of forests $\cF$ and the groupoid of all symmetric groups.
For more details on such products we refer the reader to the articles of Brin and of Witzel and Zaremsky \cite{Brin07,Witzel-Zaremsky18}.

{\bf Elements of $V$ as fractions.}
Any element of $V$ is an equivalence class of a pair of symmetric trees $\dfrac{\tau\circ t}{\sigma\circ s}.$
Observe that $\dfrac{\tau\circ t}{\sigma\circ s} = \dfrac{\sigma^{-1}\circ \tau\circ t}{s}.$
Hence, any element of $V$ can be written as $\dfrac{\sigma\circ t}{s}$ for some trees $t,s$ and permutation $\sigma.$
Note that formally the fraction $\dfrac{\tau\circ t}{\sigma\circ s}$ is equal to the signed path of morphisms $(\tau\circ t)^{-1} \circ (\sigma\circ s)=t^{-1}\circ \tau^{-1} \circ \sigma\circ s.$

{\bf Affine forests and Thompson's group $T$.}
Let $\Z/m\Z$ be the cyclic group of order $m$ identified as a subgroup of the symmetric group $S_m$ and consider the subcategory $\cAF\subset \cSF$ of \textit{affine} forests where 
$$\cAF(n,m)=\cF(n,m)\times \Z/m\Z.$$
It is a category admitting a calculus of left-fractions and the group of fractions associated to the objet $1$ is isomorphic to Thompson's group $T$.
We will often identify $\cF$ and $\cAF$ as subcategories of $\cSF$ giving embeddings at the group level $F\subset T\subset V$.

{\bf Reduced pair.}
We say that a pair of symmetric trees $( \tau\circ t  , \sigma\circ s )$ is \textit{reduced} if there are no other pairs $( \tau'\circ t' , \sigma'\circ s')$ in the same class such that $t'$ has strictly less leaves than $t$.

{\bf Monoidal structure.}
We equipped $\cSF$ with a monoidal structure $\ot$ that is 
$$n\ot m:=n+m$$ for objects $n,m$ and the tensor product of two symmetric forests 
$$(\sigma\circ f) \ot (\sigma'\circ f')=(\sigma\otimes\sigma')\circ (f\ot f')$$ 
consists in concatenating the two diagrams horizontally such that $(\sigma\circ f)$ is placed to the left of $(\sigma'\circ f').$
If we consider the tree and forest $t,f$ of above, then 
$$t\otimes f = \funfun \quad \forest \ .$$
This monoidal structure of $\cSF$ confers a monoidal structure on the smaller category $\cF$ but not on $\cAF$ as a product of cyclic permutations is in general not a cyclic permutation.

\begin{remark}Note that the common definition of a monoidal or tensor category demands that $\ot$ has a neutral element. Here, this can be added by considering the object $0$ and the empty diagram playing the role of $\id_0$.\end{remark}

{\bf Metric.}
We equip forests with the usual metric. Hence, an edge between two vertices if of length one. Now, recall that by convention the trivial tree $I$ has one root and one leaf that are equal and thus is of diameter zero. If $Y$ is the tree with two leaves, then each of its leaf is at distance one from the root. If we consider the tree $t=\funfun$, then its first leaf is a distance two from the root and the second and third leaves are at distance two and one from the root, respectively.

{\bf Order.}
We equip $\cF$ with a partial order $\leq$ defined as follows: 
$$s\leq t \text{ if there exists $f$ satisfying } t=f\circ s.$$
Note that if $s,t$ are trees, then $s\leq t$ if and only if $s$ is a rooted subtree of $t$. 
Moreover, the set of trees equipped with $\leq$ is directed, i.e.~for all trees $s,t$ there exists a third tree $z$ satisfying that $s\leq z$ and $t\leq z$.

\subsubsection{Classical actions of the Thompson's groups on the unit interval}\label{sec:VactQ}
We present the usual action of $V$ on the unit interval which explains the correspondence between trees and certain partitions of the unit interval.
Additional details can be found in \cite{Cannon-Floyd-Parry96}.

{\bf Standard dyadic interval and partition.}
Consider the infinite binary rooted tree $t_\infty$ and decorate its vertices by intervals such that the root corresponds to the half-open interval $[0,1)$ and the successors of a vertex decorated by $[d,d')$ are decorated by $[d,\tfrac{d+d'}{2})$ to the left and $[\tfrac{d+d'}{2} , d' )$ to the right.
\newcommand{\tinf}{
\begin{tikzpicture}
\node at (0,-.25) {$[0,1)$};
\draw (0,0)--(-2,1);
\draw (0,0)--(2,1);
\node at (-2,1.25) {$[0,1/2)$};
\draw (-2,1.5)--(-3,2.5);
\draw (-2,1.5)--(-1,2.5);
\node at (2,1.25) {$[1/2,1)$};
\draw (2,1.5)--(1,2.5);
\draw (2,1.5)--(3,2.5);
\node at (-3,2.75) {$[0,1/4)$};
\node at (-1,2.75) {$[1/4,1/2)$};
\node at (1,2.75) {$[1/2,3/4)$};
\node at (3,2.75) {$[3/4,1)$};
\node at (-3,3.25) {$\cdots$};
\node at (-1,3.25) {$\cdots$};
\node at (1,3.25) {$\cdots$};
\node at (3,3.25) {$\cdots$};
\end{tikzpicture}}
Here is the beginning of this labelled tree:
$$\tinf\ .$$
Intervals appearing in this tree are called {\it standard dyadic intervals} and form the set
$$\{ [\tfrac{a}{2^n}, \tfrac{a+1}{2^n}) : n\geq 0 , 0\leq a \leq 2^n-1 \}.$$
Consider a tree $t\in\fT$ and write $I_n$ for the interval corresponding to the $n$-th leaf of $t$ where $t$ is viewed as a rooted subtree of $t_\infty.$
We have that $\{I_1,\cdots,I_n\}$ is a partition of $[0,1)$ that we call a {\it standard dyadic partition}.

{\bf Action of $V$ on the unit torus.}
Now consider $g=\dfrac{\tau\circ t}{\sigma\circ s} \in V$ and the standard dyadic partitions $\{I_1,\cdots,I_n\}$ and $\{J_1,\cdots,J_n\}$ of $[0,1)$ associated to the trees $s$ and $t$ respectively. The element $g$ acting on $[0,1)$ is the unique piecewise linear function with positive constant slope on each $I_k$ that maps $I_{\sigma^{-1}(i)}$ onto $J_{\tau^{-1}(i)}$ for any $1\leq i\leq n.$
From this description of $V\act [0,1)$ we easily deduce that $T$ is the group of homeomorphisms of the unit torus that is piecewise affine with slopes powers of 2 and finitely many breakpoints while $F$ is the subgroup of $T$ fixing $0$ (and thus acting on $[0,1]$ by homeomorphisms).

{\bf Action of $V$ on the dyadic rationals.}
Put $\Q_2$ the set of dyadic rational in $[0,1)$ and observe that the action of $V$ on $[0,1)$ restricts to an action on $\Q_2$. 
This action will appear in the construction of the wreath product $\oplus_{\Q_2}\Ga\rtimes V$ of the main theorem.
Note that the action $V\act \Q_2$ is conjugated to the homogeneous action of $V\act V/V_{1/2}$ where $V_{1/2}$ is the stabiliser subgroup of the point $1/2$.

\subsection{Jones' actions}\label{sec:actions}

\subsubsection{General case}
Consider a small category $\cC$ admitting a calculus of left-fractions in a fixed object $e$, another category $\cD$ whose objects are {\it sets}, and a covariant functor $\Phi:\cC\to\cD$.
Consider the set of morphisms with source $e$ that we equip with the following order:
$$t\leq s \text{ if there exists $f$ satisfying } s=f\circ t.$$ 
This is the generalisation of the order we put on the set of trees at the end of Section \ref{sec:forest}.
Note that it is a directed set precisely because $\cC$ satisfies Ore's condition in $e$.
Given $t\in\cC(e,b)$, we form the set $X_t$ a copy of $\Phi(b)$ and consider the directed system $(X_t : t \text{ a morphism with source } e)$ with maps $\iota_t^{ft}:X_t\to X_{ft}$ given by $\Phi(f)$.
Let $\mathscr X$ be the inductive limit that we write $\varinjlim_{t, \Phi}X_t$ to emphasize the role of $\Phi.$
It can be described as $$\{ (t, x): t\in\cC(e,b), x\in\Phi(b),b\in\ob(\cC) \} / \sim$$
where $\sim$ is the equivalence relation generated by 
$$(t,x)\sim( ft,\Phi(f)(x)).$$
We often denote by $\dfrac{t}{x}$ the equivalence class of $(t,x)$ and call it a {\it fraction}.

\begin{definition}
Let $G_\cC$ be the group of fractions of $\cC$ at the object $e$.The {\it Jones action} $\pi_\Phi:G_\cC\act \mathscr X$ associated to the functor $\Phi:\cC\to\cD$ is defined by the following formula:
$$\pi_\Phi\left( \dfrac{t}{s} \right)  \dfrac{r}{x} := \dfrac{pt}{ \Phi(q)(x)} \text{ for $p,q$ satisfying } ps=qr.$$
\end{definition}

One can check that this formula does not depend on the choice of $p,q$ and thus the action is well-defined.

\begin{remark}
\begin{enumerate}
\item When $\cC$ is right-cancellative at $e$ and $t\leq s$, then there exists a unique $f$ satisfying $s=ft$. 
Although, when $\cC$ is only {\it weak} right-cancellative at $e$, then there may be several $f$ satisfying $s=ft$. 
We still obtain a directed system but to stay fully rigorous we should write $\iota_{t,f}$ rather than $\iota_t^s$ since there may be several maps going from $X_t$ to $X_s$.
\item Note that if $\cC$ admits a calculus of left-fractions (at any objects), then we can adapt the construction and obtaining an action of the whole groupoid of fractions, see Section \ref{sec:groupoidap}.
\item If we replace $X_t$ by the set of morphisms $\cD(\Phi(e),\Phi(\target(t)))$ in the construction, then we no longer need to assume that the objects of the category $\cD$ are sets. 
This was the original definition of Jones \cite{Jones16-Thompson}.
\item A similar construction can be done for contravariant functors $\Phi:\cC\to\cD$ leading to an action of $G_\cC$. 
Formally, this makes no difference since we may consider the opposite category of $\cD$ and recovering a covariant functor. Although, in practice we will obtain inverse systems and limits rather than direct systems and colimits.
For instance, if $\cD$ is the category of finite groups, then a covariant functor will typically provide an amenable discrete group while a contravariant functor will provide a profinite group.
\end{enumerate}
\end{remark}

\subsubsection{The Hilbert space case: representations and coefficients}
Let $\cD=\Hilb$ be the category of complex Hilbert spaces with linear isometries for morphisms.
Consider a functor $\Phi:\cC\to\Hilb$.
We often write $\fH_t=X_t$ for the Hilbert space associated to $t\in\cC(e,b)$.
The inductive limit has an obvious pre-Hilbert space structure that we complete into a Hilbert space and denote by $\scrH_\Phi=\varinjlim_{t,\Phi}\fH_t$.
The Jones action $\pi_\Phi:G_\cC\act \scrH_\Phi$ is a unitary representation that we call a {\it Jones' representation.}

Let $\fH$ be the Hilbert space $\Phi(e)$ associated to the chosen object $e$ that we consider as the subspace $\fH_{\id}$ of $\scrH_\Phi$ where $\id\in\cC(e,e)$ is the identity morphism.
Note that if $\xi$ is a vector of $\fH$ and $g=\dfrac{t}{s}\in G_\cC$ is a fraction, then
\begin{equation}\label{eq:coefdef}\langle \pi_\Phi\left(\dfrac{t}{s}\right) \xi, \xi\rangle = \langle \Phi(s)\xi, \Phi(t)\xi\rangle.\end{equation}
We will be considering exclusively those kind of coefficients that can be easily computed if one understand well the functor $\Phi.$
In particular, if $\Phi(n)$ is a space constructed via a planar algebra, like in \cite{Jones17-Thompson,ABC19,Jones21}, then the coefficient of above can be computed using the skein theory of the planar algebra giving us an explicit algorithm, see also \cite{Ren18-Thompson,GolanSapir15}.

\subsubsection{The group case}
Let $\cD=\Gr$ be the category of groups and consider a functor $\Phi:\cC\to\Gr.$
We often write $\Ga_t=X_t$ for the group associated to a morphism $t\in\cC(e,b).$
The inductive limit $\varinjlim_{t,\Phi}\Ga_t$ is usually denoted $\scrG_\Phi$ and has a group structure. 
Moreover, the Jones' action $\pi_\Phi:G_\cC\act \scrG_\Phi$ is an action by group automorphisms.
We equipped $\Gr$ with the monoidal structure $\ot$ such that $\Ga_1\ot\Ga_2$ is the direct sum of these groups. If $\sigma_i:\Ga_i\to\Lambda_i,i=1,2$ are group morphisms, then $\sigma_1\ot\sigma_2$ is the following group morphism $$\Ga_1\oplus\Ga_2\ni(g_1,g_2) \mapsto (\sigma_1(g_1),\sigma_2(g_2))\in \Lambda_1\oplus\Lambda_2.$$ 
Functors of this form were first considered by Stottmeister and the author in \cite{Brot-Stottmeister-M19,Brot-Stottmeister-P19}.
A systematic study of the semi-direct product of groups $\scrG_\Phi\rtimes G_{\cC}$ has been initiated in \cite{Brothier22,Brothier21}.

\subsubsection{Monoidal functors}\label{sec:monoidal}
We will mainly consider covariant monoidal functors from the category of forests $\cF$ into $\Hilb$ or $\Gr.$
On $\Hilb$ we consider in this article the classical monoidal structure $\ot$ so that $\ell^2(I)\ot \ell^2(J)\simeq \ell^2(I\times J)$.
Observe that an elementary forest $f_{i,n}$ decomposes as follows
$$I^{\ot i-1} \ot Y \ot I^{n-i}.$$
If $\Phi:\cF\to\cD$ is a monoidal functor, then 
$$\Phi(n)=\Phi(1)^{\ot n}$$ and 
$$\Phi(f_{i,n}) = \id^{\ot i-1}\ot \Phi(Y)\ot \id^{n-i}.$$
Since any forest is the composition of some $f_{i,n}$ we obtain that $\Phi$ is completely characterized by the objet $\Phi(1)$ and the morphism $\Phi(Y):\Phi(1)\to \Phi(1)\ot\Phi(1).$
When $\cD=\Hilb$ we may use the following notations: $\fH:=\Phi(1)$ and $R:=\Phi(Y).$
In that case $R:\fH\to \fH\ot\fH$ is a linear isometry.

If $\cD=\Gr$, then we may adopt the notations: $\Xi:\cF\to\Gr$ with $\Ga:=\Xi(1)$ and $S:=\Phi(Y)$. 
Hence, $S:\Ga\to\Ga\oplus\Ga$ is a group morphism.

Given a monoidal functor $\Phi:\cF\to\cD$ we have a Jones' action $\pi_\Phi:F\act \mathscr X.$
Assume that $\cD$ is a symmetric category like $\Hilb$ and $\Gr$.
We can then extend this action into an action of the larger Thompson's group $V$ via the formula
\begin{equation}\label{eq:JactionV}\dfrac{\theta\circ t}{\sigma\circ s} \cdot \dfrac{s}{x} := \dfrac{t}{\Tens(\theta^{-1} \sigma)x}, \text{ where } \Tens(\kappa)(x_1\ot\cdots\ot x_n) = x_{\kappa^{-1}(1)}\ot\cdots \ot x_{\kappa^{-1}(n)}.\end{equation}
When $\cD=\Hilb$, then the formula \eqref{eq:coefdef} becomes:
$$\langle \pi_\Phi\left(\dfrac{\theta\circ t}{\sigma\circ s} \right) \xi, \xi\rangle = \langle \Tens(\sigma)\Phi(s)\xi, \Tens(\theta)\Phi(t)\xi\rangle$$
for $\xi\in\Phi(1).$

Here is another interpretation of the extension of the Jones action to Thompson's group $V$.
We extend the monoidal functor $\Phi:\cF\to\cD$ uniquely into a monoidal functor $\ov\Phi:\cSF\to\cD$ satisfying $\ov\Phi(1)=\Phi(1), \ov\Phi(Y)=\Phi(Y)$ and where $\ov\Phi(\sigma)=\Tens(\sigma)$ for a permutation $\sigma.$
We then perform the Jones construction applied to $\ov\Phi.$ We have an inductive limit of spaces $\fH_{\sigma\circ t}$ where now Hilbert spaces are indexed by pairs $(\sigma,t)$ with $t$ a tree and $\sigma$ a permutation. 
Observe that $\fH_{\sigma\circ t}$ embeds inside $\fH_t$ via $\ov\Phi(\sigma^{-1})$ and thus the limit Hilbert space for the functor $\ov\Phi$ can be canonically identified with the one of $\Phi$ since any morphism of $\cSF$ with source $1$ (a symmetric tree) is smaller than a morphism of $\cF$ with source $1$ (a tree), i.e.~the set of trees is cofinal inside the directed set of symmetric trees.
The Jones action for $\ov\Phi$ of the larger group of fractions $G_{\cSF}$ satisfies that
$$\pi_{\ov\Phi}\left( \dfrac{\theta\circ t}{\sigma\circ s} \right) \dfrac{s}{x} = \dfrac{\sigma^{-1}\theta t}{x} = \dfrac{t}{\ov\Phi(\theta^{-1}\sigma)x} = \dfrac{t}{\Tens(\theta^{-1}\sigma)x}$$
as in \eqref{eq:JactionV}.

\subsection{Construction of larger groups of fractions}\label{sec:largergroups}
This section explains how to achieve step 2 described in the introduction: given a functor $\Xi:\cF\to\Gr$ we construct a category $\cC_\Xi$ whose group of fractions is isomorphic to the semi-direct product $\scrG\rtimes V$ where $V\act \scrG$ is the Jones action induced by $\Xi$.

\subsubsection{Larger groups of fractions}\label{sec:largercat}

{\bf A functor gives an action.}
Consider a group $\Ga$, a group morphism $S:\Ga\to\Ga\oplus\Ga$, and the unique monoidal functor $\Xi:\cF\to\Gr$ satisfying that $\Xi(1)=\Ga$ and $\Xi(Y)=S.$
Set $\scrG:= \lim_{t\in\fT, \Xi} \Ga_t$ the inductive limit group with respect to (w.r.t.) this functor where 
$$\Ga_t := \{ (g,t) ,\ g\in \Xi(\target(t))\}$$ 
is isomorphic to $\Ga^n$ when $t$ is a tree with $n$ leaves.
Intuitively, $\Ga_t$ can be interpreted as all possible decorations of the leaves of $t$ with elements of $\Ga.$
We have a Jones' action $\pi_\Xi:F\act \scrG$ that we extend to an action $\pi_\Xi:V\act \scrG$ as explained above.
Since $\pi_\Xi$ is an action by group automorphisms we can construct the semi-direct product $\scrG\rtimes_{\pi_\Xi} V.$

{\bf Group of fractions.}
We now show that $\scrG\rtimes_{\pi_\Xi} V$ arises naturally as a group of fractions.
Define the category $\cC:=\cC_\Xi$ with object $\N^*$ and sets of morphisms 
$$\cC(n,m):= \cF(n,m)\times S_m\times \Ga^m.$$
We interpret $\cF(n,m)$ (resp. $S_m$ and $\Ga^m$) as morphisms in $\cC(n,m)$ (resp. in $\cC(m,m)$), i.e.~a triple $(f,\sigma,g)\in \cC(n,m)$ is interpreted as a composition $g\circ \sigma\circ f.$
A morphism is identified with an isotopy class of diagrams that are vertical concatenation of forests, permutations, and a tuple of elements of $\Ga.$ 

{\bf Composition of morphisms.}
We previously explained what are the diagrams for forests and permutations and how to compose permutations with forests.
We now explain how to compose tuples of elements of $\Ga$ with forests and permutations.

An element $g=(g_1,\cdots,g_m)\in \Ga^m$ is the diagram consisting of placing $n$ dots on a horizontal line labeled from left to right by $g_1,g_2,\cdots,g_m.$
If $f\in\cF(n,m)$, then the diagram $g\circ f$ is represented by the forest $f$ whose $j$-th leaf is labeled by $g_j$.
\newcommand{\forestprime}{
\begin{tikzpicture}[baseline = .4cm]
\draw (0,0)--(0,1);
\draw (1,0)--(1,2/3);
\draw (1,2/3)--(2/3,1);
\draw (1,2/3)--(4/3,1);
\end{tikzpicture}
}
\newcommand{\gforest}{
\begin{tikzpicture}[baseline = .4cm]
\draw (0,0)--(0,1);
\draw (1,0)--(1,2/3);
\draw (1,2/3)--(2/3,1);
\draw (1,2/3)--(4/3,1);
\node at (0,1.2) {$g_1$};
\node at (2/3,1.2) {$g_2$};
\node at (4/3,1.2) {$g_3$};
\end{tikzpicture}
}
If $f= \ \forestprime$ and $g=(g_1,g_2,g_3)$, then 
$$g\circ f = \ \gforest \ .$$
If $p\in\cF(m,k)$ is another forest, then the diagram $p\circ g$ is represented by the forest $p$ whose $j$-th root is labeled by $g_j$.
\newcommand{\forestg}{
\begin{tikzpicture}[baseline = .4cm]
\draw (0,0)--(0,1);
\draw (1,0)--(1,2/3);
\draw (1,2/3)--(2/3,1);
\draw (1,2/3)--(4/3,1);
\draw (2,0)--(2,1/3);
\draw (2,1/3)--(5/3,1);
\draw (2,1/3)--(7/3,1);
\draw (13/6,2/3)--(2,1);
\node at (0,-.2) {$g_1$};
\node at (1,-.2) {$g_2$};
\node at (2,-.2) {$g_3$};
\end{tikzpicture}
}
For example, if $p=\forest$, then
$$ p\circ g = \ \forestg \ .$$
Now, we can lift up the $g_i$'s on top of the forest $p$ by applying the functor $\Xi$.
We obtain that
$$p\circ g = \Xi(p)(g)\circ p.$$
The element $\Xi(p)(g)$ is an element of $\Ga^6$ which decorates the six leaves of the forest $p$.
This process shows that a forest (here $p$) with roots decorated by elements of $\Ga$ is equal to the same forest with now its leaves decorated by elements of $\Ga$.

Formally, the rules of compositions are:
\begin{align*}
f\circ g & := \Xi(f)(g) \circ f, \ \forall f\in\cF(n,m), g\in \Ga^n\\
\sigma\circ (g_1,\cdots,g_n) & = (g_{\sigma^{-1}(1)},\cdots,g_{\sigma^{-1}(n)}) \circ \sigma, \ \forall g_i\in \Ga, \sigma\in S_n\\
\end{align*}

This indeed defines associative compositions for morphisms and provides a categorical structure to $\cC.$ 
Define a monoidal structure $\ot$ on $\cC$ such as $n\ot m:= n+m$ for objects and the tensor product of morphisms corresponds to horizontal concatenation from left to right as in $\cSF.$
The following proposition follows from the definitions of calculus of left-fractions.

\begin{proposition}\label{prop:groupCat}
The category $\cC$ admits a calculus of left-fractions. Its group of fractions $G_\cC$ associated to the object $1$ is isomorphic to the semi-direct product $\scrG\rtimes_{\pi_\Xi} V$ constructed via the functor $\Xi:\cF\to\Gr.$
\end{proposition}

\begin{proof}
The two axioms of calculus of left-fractions are trivially satisfied by $\cC.$
Let us build an isomorphism from $\scrG\rtimes_{\pi_\Xi} V$ to $G_\cC$.
Consider $v\in V$ and $g\in\scrG.$
There exists a large enough tree $t$ such that $v=\dfrac{t}{\sigma s}$ and $g\in \Ga_t$ where $s$ is another tree and $\sigma$ a permutation.
To emphasise that we consider the representative of $g$ inside $\Ga_t$ we write $g$ as a fraction $\dfrac{t}{g_t}.$
Define the family of maps: 
$$P_t:(\dfrac{t}{\sigma s} , \dfrac{t}{g_t})\mapsto \dfrac{g_t t }{\sigma s}.$$
Those maps are compatible with the directed systems associated to $V,\scrG,$ and $G_\cC.$
Indeed if $f$ is a (symmetric) forest, then $\dfrac{t}{\sigma s} = \dfrac{ft}{f\sigma s}$ and $\dfrac{t}{g_t} = \dfrac{ft}{\Xi(f)(g_t)}.$
Our maps satisfy the following: 
$$P_{ft}(\dfrac{ft}{f\sigma s} , \dfrac{ft}{\Xi(f)(g_t)}) = \dfrac{ \Xi(f)(g_t)f t }{f\sigma s}=\dfrac{ fg_t t }{f\sigma s}=\dfrac{g_t t }{\sigma s} =P_t(\dfrac{t}{\sigma s} , \dfrac{t}{g_t}).$$
The limit map $\varinjlim_t P_t$ defines a group isomorphism from $\scrG\rtimes_{\pi_\Xi} V$ onto $G_\cC.$
\end{proof}

{\bf Fractions.}
Every element of $V$ can be written as a fraction $\dfrac{\sigma t}{s}$ where $t,s$ are trees with the same number of leaves and $\sigma$ is a permutation.
Similarly, using composition of morphisms inside the category $\cC_\Xi$, we observe that any element of $G_\cC$ can be written as a fraction $\dfrac{\sigma g t }{s}=\dfrac{gt}{\sigma^{-1} s}$ like in $V$ but where we labeled the leaves of $t$ with elements of the group $\Ga.$ 

\begin{remark}\label{rem:LargeCat}
We have explained how to construct a category $\cC_\Xi$ from a functor $\Xi:\cF\to\Gr$ starting from the category of forests such that the group of fractions of $\cC_\Xi$ is isomorphic to the semi-direct product obtained from the Jones action induced by $\Xi$.
This process is very general and we can replace the category $\cF$ by any other small category $\cD$ admitting a calculus of left-fractions at a certain object $e\in\ob(\cD)$.
Indeed, consider a functor $\Xi:\cD\to\Gr$ and the associated Jones' action $\al_\Xi:G_\cD\act \scrG_\Xi$ where $G_\cD$ is the group of fractions of $(\cD,e)$.
Define a new category $\cC_\Xi$ with object $\ob(\cC_\Xi)=\ob(\cD)$ and morphisms $\cC_\Xi(a,b) = \cD(a,b)\times \Xi(b)$ for $a,b$ objects. 
As before we identify $\cD(a,b)$ and $\Xi(b)$ as morphisms of $\cC_\Xi$ from $a$ to $b$ and from $b$ to $b$ respectively.
The composition of morphisms of $\cC_\Xi$ are defined such that 
$$f\circ g = \Xi(f)(g)\circ f, \text{ for } f\in\cD(a,b), g\in\Xi(a), a,b\in\ob(\cC_\Xi).$$
One can check that $\cC_\Xi$ is a small category admitting a calculus of left-fractions at $e$ whose associated group $G_{\cC_\Xi}$ is isomorphic to the semi-direct product $\scrG_\Xi\rtimes G_\cD$.

In particular, we can choose to replace permutations by braids and obtaining braided versions of our groups. This produces wreath product where the braided Thompson group is acting rather than $V$.
\end{remark}

\begin{notation}
We often write $v$ for an element of $V$, $g$ for an element of $\Ga$ or $\Ga^n$ and $v_g$ for an element of $G_\cC.$
\end{notation}

{\bf Extending Jones' actions to larger categories.}
We explain how to extend a Jones' action to a larger category.
Assume we have a monoidal functor $\Phi:\cF\to\cD$ into a symmetric category.
This defines a Jones' action $\pi:F\act \mathscr X$ that can be extended to an action of $V$ as we saw in Section \ref{sec:monoidal}. 
Let us explain how this same process allow us to extend $\pi$ to an action of the even larger group $G_\cC$ where $\cC=\cC_\Xi.$
Write $X:=\Phi(1)$ and assume we have an action by automorphisms $\rho:\Ga\act X$. 
We extend $\pi$ to the group of fractions $G_\cC$ such as:
\begin{equation}\label{eq:largecat}\pi\left( \dfrac{g\sigma t}{s} \right) \dfrac{s}{x} = \dfrac{t}{\Tens(\sigma^{-1})\rho^{\ot n}(g^{-1})x}\end{equation}
for $t,s$ trees with $n$ leaves, $\sigma\in S_n$ and $g\in \Ga^n$.

Formula \ref{eq:largecat} can be obtained as follows.
Extend the functor $\Phi$ into a functor $\ov\Phi:\cC\to\cD$ such that $\ov\Phi(1)=\Phi(1), \ov\Phi(Y)=\Phi(Y)$ and $\ov\Phi(\sigma)=\Tens(\sigma), \ov\Phi(g)=\rho(g), \sigma\in S_n, g\in \Ga.$
We observe that for any morphism $g\sigma t$ of $\cC$ with source $1$ we have that $g\sigma t\leq t$ and thus we can identify the inductive limit $\mathscr X$ obtained with $\Phi$ with the inductive limit obtained with $\ov\Phi.$
Therefore, $$\pi\left( \dfrac{g\sigma t}{s} \right) \dfrac{s}{x} = \dfrac{g\sigma t}{x} = \dfrac{(g\sigma)^{-1} g\sigma t}{\ov\Phi((g\sigma)^{-1})x} = \dfrac{t}{\Tens(\sigma^{-1}) \rho^{\ot n}(g^{-1}) x}$$
which recovers Formula \ref{eq:largecat}.

\subsubsection{Isomorphism with a wreath product}\label{sec:WP}
We end this subsection by giving a precise description of $G_\cC$ for a specific choice of functor.
Let $V\act \Q_2$ be the restriction of the usual action of $V$ on the unit interval to the dyadic rationals $\Q_2$, see Section \ref{sec:VactQ} for details.
Let $\Ga$ be a group and $\theta\in\Aut(\Ga)$ an automorphism of $\Ga.$
Given $v\in V$ and $x\in \Q_2$ we write $v'(x)$ for the right-derivative of $v$ at $x$.
Moreover, we denote by $\log_2$ the logarithm in base $2$ so that $\log_2(2^n)=n$ for all $n\in\Z.$
Consider the direct sum $\oplus_{\Q_2}\Ga$ of all maps $a:\Q_2\to\Ga$ that are finitely supported and define the actions $$V\act \oplus_{\Q_2}\Ga, (v\cdot a)(x):= \theta^{\log_2(v'(v^{-1}x))}(a(v^{-1}x)),\ v\in V, a\in \oplus_{\Q_2}\Ga, x\in \Q_2.$$
We write $$\Ga\wr_{\Q_2}^\theta V:=\oplus_{\Q_2}\Ga\rtimes^\theta V$$
for the associated semi-direct product that we call a {\it twisted wreath product}.
When $\theta=\id$ is the identity we drop the superscript $\theta$ and say that we have a wreath product or an {\it untwisted} wreath product.
Here is a key observation that was done in \cite[Section 4.2]{Brothier22}.
\begin{proposition}\label{prop:Bernoulli}
Fix a group $\Ga$ and an automorphism $\theta\in\Aut(\Ga)$.
Consider the unique covariant monoidal functor $\Xi:\cF\to\Gr$ satisfying
$$\Xi(1)=\Ga \text{ and } \Xi(Y)(g)=(\theta(g),e) \text{ for all } g\in \Ga.$$
Denote by $\scrG:=\varinjlim_{t\in \fT, \Xi}\Ga_t$ the limit group obtained and by $\pi_\Xi:V\act \scrG$ the Jones action.
There is a group isomorphism from $\scrG$ onto $\oplus_{\Q_2} \Ga$ that intertwines the Jones action $\pi_\Xi:V\act \scrG$ and the twisted action $V\act \oplus_{\Q_2} \Ga$ described above.
In particular, the group of fractions $G_\cC$ associated to the larger category $\cC:=\cC_\Xi$ is isomorphic to the twisted wreath product $\Ga\wr_{\Q_2}^\theta V$.
\end{proposition}

Note that it is easy to understand graphically the composition of morphisms in the category $\cC_\Xi$ associated to the specific functor $\Xi$ of Proposition \ref{prop:Bernoulli}.
Indeed, $Y \circ g = (\theta(g),e) \circ Y$ for any $g\in \Ga.$ 
Hence, elements of $\Ga$ can go up in a tree by going to the left and by adding some trivial elements $e$ to their right.
\newcommand{\Xigforest}{
\begin{tikzpicture}[baseline = .4cm]
\draw (0,0)--(0,1);
\draw (1,0)--(1,2/3);
\draw (1,2/3)--(2/3,1);
\draw (1,2/3)--(4/3,1);
\draw (8/3,0)--(8/3,1/3);
\draw (8/3,1/3)--(6/3,1);
\draw (8/3,1/3)--(10/3,1);
\draw (9/3,2/3)--(8/3,1);
\node at (0,1.2) {$g_1$};
\node at (2/3,1.2) {$\theta(g_2)$};
\node at (4/3,1.2) {$e$};
\node at (6/3,1.2) {$\theta(g_3)$};
\node at (8/3,1.2) {$e$};
\node at (10/3,1.2) {$e$};
\end{tikzpicture}
}
For example, if $g=(g_1,g_2,g_3)$ and $$f = \forest \ ,$$ then 
$$f\circ g = \ \forestg \ = \Xi(f)(g)\circ f = \Xigforest \ .$$

\section{Haagerup property for Thompson's groups $F$ and $T$}\label{sec:FT}
In this article we prove that certain wreath products have the Haagerup property.
This result is new and is done by using the original definition of the Haagerup property: there exists a net of positive definite maps vanishing at infinity that converges pointwise to 1.
The construction of the net is done using Jones' technology and by identifying wreath products with certain groups of fractions.
We could give a single proof. However, for pedagogical reasons we will give five of them with increasing level of technicality. 
More precisely, we provides proofs for the following results:
\begin{enumerate}
\item Thompson's group $F$ has the Haagerup property;
\item Thompson's group $T$ has the Haagerup property;
\item Thompson's group $V$ has the Haagerup property;
\item If $\Ga$ has the Haagerup property, then so does the wreath product $\Ga\wr_{\Q_2}V$;
\item If $\Ga$ has the Haagerup property and $\theta\in\Aut(\Ga)$ is any automorphism of $\Ga$, then the associated {\it twisted} wreath product $\Ga\wr_{\Q_2}^\theta V$ has the Haagerup property.
\end{enumerate}
The important gaps of difficulties between these cases are from $T$ to $V$ and from $V$ to the untwisted wreath product.

\subsection{Proof for Thompson's group $F$}
Consider the Hilbert space $\fH:=\ell^2(\N)$ where $\N$ is the additive monoid of natural numbers (including zero).
We write $(\de_n:\ n\geq 0)$ for the usual orthonormal basis of $\fH$.
We identify $\fH^{\ot k}$ with $\ell^2(\N^k)$ and consider the usual orthonormal basis $(\de_x:\ x\in \N^k)$ of it for all $k\geq 1.$
Fix a real number $0\leq \al\leq 1$ and set $\beta:=\sqrt{1-\al^2}.$
We now define a linear isometry:
\begin{align*}
R_\al:&\fH\to \fH\ot \fH\\
&\de_0 \mapsto \al \de_{0,0} + \beta \de_{1,1}\\
& \de_n\mapsto \de_{n,n} \text{ for all } n\geq 1.
\end{align*}
This defines uniquely a monoidal covariant functor $\Phi_\al:\cF\to\Hilb$ and thus a Jones' representation $\pi_\al:F\act \scrH_\al$.
Now, $\fH$ embeds in $\scrH_\al$ and we may then consider $\de_0$ as a unit vector of $\scrH_\al$.
We set 
$$\phi_\al:F\to\C, \ g\mapsto \langle \pi_\al(g)\de_0,\de_0\rangle$$
our matrix coefficient which is a positive definite map.

{\bf Key fact.} Consider a tree $t$ with $n$ leaves and the list $d^t:=(d^t_1,\cdots,d^t_n)$ of distances between the root of $t$ and each of its leaf.
The map $t\mapsto d^t$ is injective. 
With this fact we will be able to easily prove the Haagerup property for $F$.

By the key fact we have that when $\al=0$, then the cyclic component of $\pi_0$ associated to the vector $\de_0$ is unitary equivalent to the left-regular representation $\la_F:F\act \ell^2(F)$.
When $\al=1$, then the cyclic component of $\de_0$ becomes unitary equivalent to the trivial representation $1_F$.
Hence, we have constructed a continuous path of representations between the trivial and the left-regular ones.

In particular, for all $g\in F$ we have that $\lim_{\al\to 1} \phi_\al(g)=1$.
To conclude that $F$ has the Haagerup property it is then sufficient to prove that for all $0<\al<1$ we have that $\phi_\al$ vanishes at infinity.
We explain briefly why this is the case.

Consider $g=\dfrac{t}{s}$ in $F$ where $t,s$ are trees with same number of leaves say $n$.
Observe that 
$$\phi_\al(g)= \langle \pi_\al(\frac{t}{s})\de_0,\de_0\rangle = \langle \Phi_\al(s)\de_e,\Phi_\al(t)\de_e\rangle.$$
The vector $\Phi_\al(s)\de_e$ belongs to $\fH^{\ot n}$ and can easily be decomposed over the usual orthonormal basis.
Indeed, for each rooted subtree $x$ of $s$ we realise the decomposition $s=f_x\circ x$ where $f_x$ is a uniquely defined forest.
The forest $f_x$ has $n$ leaves. We write $d_j^{x,s}$ for the distance from this $j$-th leaf of $f$ to the root of $f_x$ that is in the same connected component.
We obtain that 
$$\Phi_\al(s)\de_e=\sum_x c_x \de_{d^{x,s}}$$
where $d^{x,s}$ is the multi-index $(d^{x,s}_1,\cdots,d^{x,s}_n)$ and $c_x$ a certain coefficient equal to a product of $\al$ and $\beta$.
Similarly, $\Phi_\al(t)\de_e$ admits such a decomposition into $\sum_y c_y \de_{d^{y,t}}.$
Therefore, 
$$\phi_\al(g)=\sum_{x,y} c_x c_y \langle \de_{d^{x,s}},\de_{d^{y,t}}\rangle.$$
Observe that $\langle \de_{d^{x,s}},\de_{d^{y,t}}\rangle=1$ when $d^{x,s}=d^{y,t}$ meaning that the forests $f^x$ and $f^y$ are equal by the key fact of above.

We deduce the following second key fact:
if $\dfrac{t}{s}$ is an irreducible fraction we have that all the coefficients of above are equal to zero except one: the coefficient corresponding to the subtrees $x=s$ and $y=t$ implying that $f^x=f^y=I^{\ot n}$ are trivial.
Indeed, if there would be another nonzero coefficient then there would exists proper subtrees $x\leq s, y\leq t$ so that $f^x=f^y\neq I^{\ot n}$.
This implies that $\dfrac{t}{s}$ can be reduced into $\dfrac{y}{x}$ and thus contradicting our assumption of irreducibility.
We deduce that 
$$\phi_\al(g)=\al^{2n-2}$$
for $g$ equal to an irreducible fraction made of trees with $n$ leaves.
Since there are only finitely many of those for each fixed $n$ we deduce that $\phi_\al$ vanishes at infinity for all $0\leq \al<1$ and thus $F$ has the Haagerup property.

Note that $\pi_\al$ extends canonically into a representation of $V$. 
However, $\phi_\al$ is no longer vanishing at infinity when extended to $V$ nor on the intermediated subgroup $T$.
Indeed, if $g_n=\dfrac{t_n\circ \sigma}{t_n}$ where $t_n$ is the regular tree with $2^n$ leaves all at distance $n$ from the root and $\sigma$ is a $n$-cycle, then $\phi_\al(g_n)=1$ for all $n$ and $\alpha$.

\subsection{Proof for Thompson's group $T$}

We proceed similarly than in the $F$-case. 
Instead of considering $\N$ we consider the free monoid $M=\N*\N$ in two generators $a,b$. 
We write $e$ for the trivial element of $M$.
As above we write $\fH=\ell^2(M)$ for the associated Hilbert space and $(\de_x:\ x\in M)$ for the usual orthonormal basis.
Fix $0\leq \al\leq 1$, set $\beta:=\sqrt{1-\al^2}$, and define the linear isometry:
\begin{align*}
R_\al:&\fH\to \fH\ot \fH\\
&\de_e \mapsto \al \de_{e,e} + \beta \de_{a,b}\\
& \de_x\mapsto \de_{xa,xb} \text{ for all } x\in M, x\neq e.
\end{align*}
This provides a functor $\Phi_\al$, a Jones representation $\pi_\al:T\act \scrH_\al$, and a matrix coefficient:
$$\phi_\al:T\to \C, g\mapsto \langle \pi_\al(g)\de_e,\de_e\rangle.$$

We have that the cyclic subrepresentation of $\pi_\al$ associated to the vector $\de_e$ interpolates the trivial and the left-regular representations of $T$.
To obtain the Haagerup property for $T$ it is then sufficient to show that $\phi_\al$ vanishes at infinity for all $0<\al<1.$

{\bf Key fact:} Consider a tree $t$ with $n$ leaves and $\sigma$ a cyclic permutation of $\{1,\cdots,n\}$.
We write $w^t_i$ for the (unique geodesic) path from the root of $t$ to its $i$-th leaf. We identify $w^t_i$ with a word $x_1\cdots x_k$ in the letters $a,b$ where $k$ is the length of the path and $x_j=a$ when the $j$-th edge of the path is a left-edge and $x_j=b$ otherwise.
The map $(t,\sigma)\mapsto (w^t_{\sigma(1)},\cdots,w^t_{\sigma(n)})$ is injective.

Using the key fact we can proceed similarly than above and conclude that if $g=\dfrac{t\circ \sigma}{s}\in T$ is a reduced fraction with $t,s$ trees with $n$ leaves, and $\sigma$ a cyclic permutation, then 
$$\phi_\al(g)=\al^{2n-2}.$$
This proves that $T$ has the Haagerup property.

Note that when we extend $\phi_\al$ to $V$ we no longer have a map vanishing at infinity, see \cite[Remark 1]{Brothier-Jones19}.

\section{Haagerup property for Thompson's group $V$}\label{sec:VH}

\subsection{The family of isometries, functors, representations, and matrix coefficients}\label{sec:coefV}
Consider the free monoid $M$ in the four generators $a,b,c,d$ and let $\fH:=\ell^2(M)$ be the associated Hilbert space with usual orthonormal basis $(\de_x:\ x\in M).$
Note that we use the free monoids in one, two, and four generators for constructing matrix coefficients for $F,T,$ and $V$, respectively. 

Identify $\fH^{\ot n}$ with $\ell^2(M^n)$ and thus the standard orthonormal basis of $\fH^{\ot n}$ consists in Dirac masses $\de_w$ where $w$ is a list of $n$ words in letters $a,b,c,d.$
For any real number $0\leq \al \leq 1$ we set $\beta:=\sqrt{1-\al^2}$ and define the isometry
\begin{align*}
R_\al:&\fH\to \fH\ot \fH\\
&\de_e \mapsto \al \de_{e,e} + \beta \de_{c,d}\\
& \de_x\mapsto \al\de_{xa,xb} + \beta\de_{xc,xd} \text{ for all } x\in M, x\neq e.
\end{align*}

Let $\Phi_\al:\cF\to \Hilb$ be the associated monoidal functor satisfying $\Phi_\al(1):=\fH, \Phi_\al(Y)=R_\al$ and let $\pi_\al:V\to \cU(\scrH_\al)$ be the associated Jones' representation.

Define the coefficient 
$$\phi_\al:V\to \C, \ v\mapsto \langle \pi_\al(v)\de_e,\de_e\rangle.$$
Observe that if $v=\dfrac{\sigma\circ t}{s}$, then 
\begin{equation}\label{eq:phivde}\phi_\al(v) = \langle \Phi_\al(s)\de_e, \Tens(\sigma)\Phi_\al(t)\de_e\rangle
\end{equation} where $$\Tens(\sigma) \left(\xi_1\ot\cdots\ot \xi_m \right):= \xi_{\sigma^{-1}(1)}\ot\cdots\ot \xi_{\sigma^{-1}(m)}.$$

\subsection{Interpolation between the trivial and the left-regular representations}

It is easy to see that the representations $\pi_0$ and $\pi_1$, that we restrict to the cyclic space generated by $\de_e$, are unitary equivalent to the left-regular representation $\la_V$ and to the trivial representation $1_V$, respectively.
In particular, $\lim_{\al\to 1} \phi_\al(v) = 1$ for any $v\in V$.
By definition, $\phi_\al$ is positive definite for any $\al$.
Therefore, it is sufficient to show that $\phi_\al$ vanishes at infinity for any $0<\al<1$ to prove that $V$ has the Haagerup property.

From now on we fix $0<\al<1$ and suppress the subscript $\al$ thus writing $R,\Phi,\pi,\phi$ for $R_\al,\Phi_\al,\pi_\al,\phi_\al.$

\subsection{The set of states}\label{sec:state}

Consider a tree $t\in\fT$ with $n$ leaves.
Put $\cV(t)$ the set of trivalent vertices of $t$ that is a set of order $n-1$ and let 
$$\St(t):=\{ \cV(t)\to \{0,1\}\}$$ 
be the set of maps from the trivalent vertices of $t$ to $\{0,1\}$ that we call the set of states of $t$.
Consider the maps
\begin{align*}
R(0):&\fH\to \fH\ot\fH\\ 
&\de_e \mapsto \al\de_{e,e}\\  
&\de_x \mapsto  \al\de_{xa,xb} \text{ if } x\in M, x\neq e
\end{align*}
and
\begin{align*}
R(1):&\fH\to \fH\ot\fH\\ 
&\de_e \mapsto \beta\de_{c,d}\\  
&\de_x \mapsto  \beta\de_{xc,xd} \text{ if } x\in M, x\neq e.
\end{align*}
By definition we have 
$$R=R(0)+R(1).$$

Given a state $\tau\in \St(t)$, we consider the operator $R(\tau):\fH\to \fH^{\ot n}$ defined as follows.
If $t$ decomposes as a product of elementary forests $ f_{j_{n-1},n-1}\circ f_{j_{n-2},n-2} \circ\cdots f_{j_2, 2}\circ f_{1,1}$ and if $\nu_k$ is the unique trivalent vertex of $f_{j_k,k}$, then 
$$R(\tau) = (\id^{\ot j_{n-1}-1} \ot R(\tau(\nu_{n-1})) \ot \id^{n-1-j_{n-1}} ) \circ \cdots \circ R(\tau(\nu_1)).$$

\newcommand{\fonefone}{
\begin{tikzpicture}[baseline=.8cm]
\draw (0,0)--(0,-.5);
\draw (0,0)--(-1,2);
\draw (-.5,1)--(0,2);
\draw (0,0)--(1,2);
\node at (-.8,1) {$\nu_2$};
\node at (-.3,0) {$\nu_1$};
\end{tikzpicture}
}
Here is an example:
consider the following tree with vertices $\nu_1,\nu_2:$
$$t=\fonefone.$$
If $\tau(\nu_1)=1,\tau(\nu_2)=0$, then $R(\tau) = (R(0)\ot\id)\circ R(1).$
Hence, $$R(\tau)\de_e = (R(0)\ot \id)\beta\de_{c,d}= \al\beta \de_{ca,cb,d}.$$

By definition of the functor $\Phi$ we obtain the formula
$$\Phi(t) = \sum_{\tau\in \St(t) }  R(\tau).$$
When applied to $\de_e$ we obtain:
$$\Phi(t)\de_e = \sum_{\tau\in \St(t)} \al_\tau \de_{W(t,\tau)}$$
where $\al_\tau$ is a constant depending on the state $\tau$ and $W(t,\tau)$ is a list of words of $M$ (one word per leaf).
For example, if $t$ is the tree of the figure of above, then we have four coefficients corresponding to the states taking the  values $(0,0), (0,1), (1,0),$ and $(1,1)$ at the pair of vertices $(\nu_1,\nu_2)$.
We obtain:
$$\Phi(t)\de_e =\al^2\de_{e,e,e} + \al\beta\de_{c,d,e} + \beta\al \de_{ca,cb,d} + \beta^2 \de_{cc,cd,d} .$$
If $t$ has $n$ leaves and $|\{ v\in \cV(t) : \tau(v)=0\}| = m $, then $\al_\tau = \al^{m} \beta^{n-m-1},$ the general formula being 
$$\al_\tau = \al^{ |\tau^{-1}(0)|} \beta^{ | \tau^{-1} (1) |}.$$
If $\sigma\in S_n$ is a permutation, then 
\begin{equation}\label{eq:Phivde}\Phi(\sigma\circ t)\de_e = \sum_{\tau\in \St(t)} \al_\tau \de_{\sigma W(t,\tau)},\end{equation}
where $\sigma W(t,\tau)$ is the list of words permuted by $\sigma.$

\subsection{General strategy for proving that $\phi$ vanishes at infinity}

Consider a fraction $v=\dfrac{\sigma\circ t}{t'}$. The decomposition of above provides the following:
\begin{equation}\label{eq:innerprod}
\phi(v) = \sum_{\tau\in\St(t)}\sum_{\tau'\in\St(t')} \al_\tau\al_{\tau'} \langle \de_{W(t',\tau')} , \de_{\sigma W(t,\tau)}\rangle.
\end{equation}
If $t$ has $n+1$ leaves, then the coefficient of above is a sum of $2^n\times 2^n$ inner products of vectors.
Our strategy is to prove that most of them are equal to zero when $\dfrac{\sigma t}{t'}$ is a reduced fraction, i.e.~$\sigma W(t,\tau)\neq W(t',\tau')$ for most pairs of states $(\tau,\tau')$.

Let us describe the $j$-th word $W(t,\tau)_j$ of $W(t,\tau)$.
Consider the $j$-th leaf $\ell$ of the tree $t$ and let $P_j$ be the geodesic path from the root of $t$ to this leaf. 
Denote by $\nu_1,\cdots,\nu_k$ the trivalent vertices of this path listed from bottom to top and let $e_1,\cdots, e_k$ be the edges such that the source of $e_i$ is $\nu_i$ and its target $\nu_{i+1}$ for $1\leq i\leq k-1$ while $e_k$ goes from $\nu_k$ to the leaf $\ell$.
We have 
\begin{equation}\label{eq:W}
W(t,\tau)_j = y(1) y(2)\cdots y(k) \text{ such that }
\end{equation} 
$$y(i) =  \begin{cases} e \text{ if } \tau(\nu_1)=\cdots = \tau(\nu_i)=0\\
a \text{ if $e_i$ is a left-edge and } \tau(\nu_i)=0\\
c \text{ if $e_i$ is a left-edge and } \tau(\nu_i)=1\\
b \text{ if $e_i$ is a right-edge and } \tau(\nu_i)=0\\
d \text{ if $e_i$ is a right-edge and } \tau(\nu_i)=1\\
\end{cases}$$
when in the second and fourth case we further assume that at least one of the $\tau(\nu_j)$ is equal to $1$ for $1\leq j< i$.
From this description we easily deduce the following lemma.
\begin{lemma}\label{lem:inj}
The map $\tau\in \St(t)\mapsto W(t,\tau)$ is injective.
\end{lemma}
Observe that if $r:=\max(i : \tau(\nu_s) = 0 \text{ for all } s\leq i)$, then $W(t,\tau)_j = y(r+1) y(r+2)\cdots y(k)$ with $y(r+1)=c$ or $d$.
Further, Equation \ref{eq:W} shows that the word $W(t,\tau)_j$ remembers the part of the path after the $r+1$-th vertex. 
This motivates the following decomposition.
\begin{notation}\label{not:ztau}
If $\tau$ is a state of the tree $t$, then we define $z_\tau$ to be the largest rooted subtree of $t$ satisfying that $\tau(\nu)=0$ for all (trivalent) vertices $\nu$ of $z_\tau$ (hence excluding the leaves of $z_\tau$).
Denote by $f_\tau$ the unique forest satisfying that $t=f_\tau\circ z_\tau.$
\end{notation}
{\bf Key observation:} The list of words $W(t,\tau)$ remembers the forest $f_\tau$, i.e.~if $t$ is a fixed tree and $\tau,\tau'$ are two states on two different trees $t,t'$, then $W(t,\tau)=W(t',\tau')$ implies that $f_\tau=f_{\tau'}$.

\subsection{An equivalence relation on the set of vertices}

From now on we consider an element $v\in V$ that we decompose as a fraction $v=\dfrac{\sigma\circ t}{t'}$ where $t,t'$ are trees with $n$ leaves and $\sigma$ is a permutation that we interpret as a bijection from the leaves of $t$ to the leaves of $t'$. 
We define an equivalence relation on the set of trivalent vertices of the tree $t$ which depends on the triple $(t,t',\sigma)$.

\begin{definition}\label{def:sim}
Consider two trivalent vertices $\nu,\tilde\nu$ of $t$.
Assume that there exists a trivalent vertex $\nu'$ of $t'$ and two leaves $\ell,\tilde\ell$ of $t$ that are descendant of $\nu,\tilde\nu$, respectively, and satisfying:
\begin{enumerate}
\item the leaves $\sigma(\ell)$ and $\sigma(\tilde\ell)$ are descendant of $\nu'$;
\item $d(\nu,\ell)=d(\nu',\sigma(\ell))$ and $d(\tilde\nu, \tilde\ell)=d(\nu',\sigma(\tilde\ell))$ where $d$ is the usual distance on trees.
\end{enumerate}
In that case we say that $\nu$ is equivalent to $\tilde\nu$ and write $\nu\sim \tilde\nu$.
\end{definition}

It is easy to see that $\sim$ defines an equivalence relation.
The next proposition implies that there are very few pairs of states $(\tau,\tau')$ satisfying that $W(t',\tau')=\sigma W(t,\tau)$.

\begin{proposition}\label{prop:sim}
Consider the fraction $\dfrac{\sigma\circ t}{t'}$ and a state $\tau\in\St(t)$. 
Assume that there exists a state $\tau'\in\St(t')$ such that $\sigma W(t,\tau)=W(t',\tau')$.
The following assertions are true:
\begin{enumerate}
\item The state $\tau$ is constant on equivalence classes of vertices under the relation $\sim$, i.e.~$\tau(\nu)=\tau(\tilde\nu)$ if $\nu\sim\tilde\nu$;
\item If $\nu$ is a vertex of $f_\tau$ and the fraction $\dfrac{\sigma\circ t}{t'}$ is irreducible, then there exists $\tilde\nu\neq \nu$ in $f_\tau$ such that $\tilde\nu\sim \nu$;
\item There is at most one state $\tau'\in\St(t')$ satisfying $\sigma W(t,\tau)=W(t',\tau')$. In that case we have $\al_\tau=\al_{\tau'}$.
\end{enumerate}
\end{proposition}

\begin{proof}
Proof of (1).
Consider vertices $\nu,\tilde\nu$ of $t$ that are equivalent under the relation $\sim$.
Denote by $\ell,\tilde\ell$ and $\nu'$ as in Definition \ref{def:sim}.
The equality $\sigma W(t,\tau)=W(t',\tau')$ together with Formula \ref{eq:W} imply that $\tau(\nu)=\tau'(\nu')$ and $\tau(\tilde\nu)=\tau'(\nu')$.

Proof of (2).
Assume that $\nu$ is a vertex of $f_\tau$ and that there are no other $\tilde\nu$ such that $\nu\sim \tilde \nu$.
We will show that the fraction $\dfrac{\sigma\circ t}{t'}$ is necessarily reducible.
Let $t_\nu$ be the maximal subtree of $t$ with root $\nu$.
Hence, the leaves of $t_\nu$ are all the leaves of $t$ that are descendant of $\nu$.
Note that since $\nu$ is a trivalent vertex we have that the tree $t_\nu$ has at least two leaves (and is thus nontrivial).
For each leaf $\ell$ of $t_\nu$ we consider $c_\ell$: the geodesic path from $\nu$ to $\ell$.
Consider now the leaf $\sigma(\ell)$ of $t'$ and $c_\ell'$ the geodesic path in $t'$ ending at $\sigma(\ell)$ and of same length than $c_\ell$.
The equality $\sigma W(t,\tau)=W(t',\tau')$ implies that the distance between $\ell$ and a root of $f_\tau$ is equal to the distance between $\sigma(\ell)$ and a root of $f_{\tau'}$.
Since $\nu$ is a vertex of $f_\tau$, the whole path $c_\ell$ is contained in $f_\tau$, and therefore the whole path $c_\ell'$ is contained in $f_{\tau'}.$ 
Denote by $s'$ the subgraph of $t'$ equal to the union of all the paths $c_\ell'$ where $\ell$ runs over all the leaves of $t_\nu$. We are going to show that $s'$ is a tree isomorphic to $t'$.

We claim that all the paths $c_\ell'$ starts at a common vertex $\nu'$ of $t'$.
Indeed, denote by $\cV'$ the set of all the sources of the paths $c_\ell'$. 
Let $f'\subset t'$ be the maximal subforest whose set of roots is equal to $\cV'$.
If $\ell'$ is a leaf of $f'$, then we can consider $\sigma^{-1}(\ell')$ which is a leaf of $t$.
By assumption there are no other $\ti\nu$ in $t$ that is equivalent to $\nu$.
This forces to have that $\sigma^{-1}(\ell')$ is a leaf of $t_\nu$ for all leaf $\ell'$ of $f'$.
Moreover, by repeating this argument we deduce that all leaves of $t_\nu$ must be equal to a certain $\sigma^{-1}(\ell')$ with $\ell'$ a leaf of $f'$, i.e.~$\sigma$ restricts to a bijection from the leaves of $t_\nu$ to the leaves of $f'.$
By using that $f'\subset f_{\tau'}$ and $t_\nu\subset f_\tau$ we deduce by an induction on the number of leaves of $t_\nu$ that $f'$ must be a tree that we write $t_\nu'$.
This uses that $W(t,\tau)$ remembers the forest $f_\tau$ and in particular the structure of subforests of it like $t_\nu$.
This proves the claim.
Hence, all $c_\ell'$ starts at a common vertex $\nu'$ of $t'$.

The equality $\sigma W(t,\tau)=W(t',\tau')$ together with the fact that $t_\nu\subset f_\tau$ and $f'\subset f_{\tau'}$ implies (via an easy induction on the number of leaves of $t_\nu$) that $\sigma$ respects the order of the leaves, i.e.~the $i$-th leaves of $t_\nu$ is sent by $\sigma$ to the $i$-th leaf of $t_\nu'$ for any $i.$ 
Using again the equality $\sigma W(t,\tau)=W(t',\tau')$ we deduce that the two trees $t_\nu$ and $t_\nu'$ are necessarily isomorphic (as ordered rooted binary trees).
This implies that we can reduce the fraction $\dfrac{\sigma\circ t}{t'}$ by removing $t_\nu$ and $t_\nu'$ at the numerator and denominator. Since $t_\nu$ was supposed to be nontrivial we obtain that our fraction $\dfrac{\sigma\circ t}{t'}$ is reducible, 
a contradiction.

Proof of (3).
By Lemma \ref{lem:inj} there are most one $\tau'\in \St(t')$ satisfying $\sigma W(t,\tau) = W(t',\tau').$
Let us assume we are in this situation for a fixed pair $(\tau,\tau').$
If $f_\tau$ is trivial (is a forest with only trivial trees), then $W(t,\sigma)$ is a list of trivial words and thus so does $W(t',\tau')$ implying that $f_{\tau'}$ is trivial. 
Therefore, $\al_\tau = \al^{n-1} = \al_{\tau'}$ where $n$ is the number of leaves of $t.$
Assume that $f_\tau$ is non-trivial and consider a vertex $\nu$ of $f_\tau$ that is connected to a leaf by an edge. 
Let $[\nu]$ be the equivalence class of $\nu$ w.r.t.~the relation $\sim.$
Consider all geodesic paths $c$ contained in $f_\tau$ starting at a root and ending at a leaf that are passing through an element of $[\nu]$. 
Define the images $c'$ of each of those paths inside $f_{\tau'}$ as explained in Proof of (2) and put $W$ the set of all last trivalent vertices (i.e.~trivalent vertices connected to a leaf) of paths $c'$.
It is easy to see that $W$ is equal to an equivalence class $[\nu']$ for a certain vertex $\nu'$ of $f_{\tau'}.$
The definition of the equivalence relation $\sim$ implies that $\sigma$ restricts to a bijection from the set of leaves that are descendant of vertices in the class $[\nu]$ to the set of leaves that are descendant of vertices in the class $[\nu']$.
The order of the class $[\nu]$ is equal to the number of leaves that are children of vertices in $[\nu]$ divided by two and thus $[\nu]$ and $[\nu']$ have same order. 
By (1), we have that the states $\tau$ and $\tau'$ take a unique value ($0$ or $1$) for any element of $[\nu]$ and $[\nu']$ that is $\tau(\nu)=\tau'(\nu').$
Consider the forests $\tilde f,\tilde f'$ that are the subforests of $f_\tau,f_{\tau'}$ obtained by removing the set of vertices $[\nu],[\nu']$ and edges starting from them, respectively. 
By applying our process to $\tilde f,\tilde f'$ we are able to show that $\al(f_\tau,\tau) = \al(f_{\tau'},\tau')$ where $\al(f_\tau,\tau)=\al^A \beta^B$ for $A$ (resp. $B$) the number of vertices of $f_\tau$ for which $\tau$ takes the value $0$ (resp. $1$).
The forest $f_\tau$ and $f_{\tau'}$ have necessarily the same number of vertices and thus so does $z_\tau$ and $z_{\tau'}.$
Since $\al_\tau = \al(f_\tau,\tau) \al^{N}$ where $N$ is the number of vertices of $z$, we obtain that $\al_\tau=\al_{\tau'}.$
\end{proof}

\subsection{Splitting the sum over rooted subtrees.}
We further decompose the sum 
$$\Phi(t)\de_e = \sum_{\tau\in\St(t)}\al_\tau \de_{W(t,\tau)}$$ by using rooted subtrees of $t.$
Let $E(t)$ be the set of all rooted subtrees of $t$ (including the trivial one and $t$).
For any $z\in E(t)$ we write $\St(t,z)$ for the set of states $\tau$ satisfying $z_\tau = z$, see Notation \ref{not:ztau}.
We obtain the following decomposition:
\begin{equation}\label{eq:phitone}
\Phi(t)\de_e = \sum_{z\in E(t)} \sum_{\tau\in \St(t,z)} \al_\tau \de_{W(t,\tau)}.
\end{equation}
Given $z\in E(t)$ we consider the unique forest $f=f_z$ satisfying that $t=f_z\circ z$. 
Fix a state $\tau\in \St(t,z)$.
For any trivalent vertex $\nu$ of $z$ we have that $\tau(\nu)=0$ and there are $n(z)-1$ of them if $n(z)$ denotes the number of leaves of $z$.
If a leaf $\nu$ of $z$ is a trivalent vertex of $t$ (i.e.~is not a leaf of $t$), then necessarily $\tau(\nu)=1$ by maximality of $z=z_\tau$. Let $b(z)$ be the number of those. 
Then $\tau$ can take any values on the other vertices of $t$, that are the vertices of $f$ that are not leaves of $z$ (trivalent vertices of $f$ that are not roots of $f$). Note that there are $n(t)-n(z)-b(z)$ such vertices and we set $m(z)$ this number and $\cV_1(f)$ those trivalent vertices.
We obtain the formula:
$$\al_\tau = \al^{n(z) -1 } \beta^{b(z)} \al_{1,\tau}(f)$$ where $\al_{1,\tau}(f)$ is a monomial in $\al,\beta$ of degree $m(z)$ that only depends on the restriction $\tau\vert_{\cV_1(f)}$.

\newcommand{\ttwofone}{
\begin{tikzpicture}[baseline=.8cm]
\draw (0,0)--(0,-.5);
\draw (0,0)--(-1,2);
\draw (-.5,1)--(0,2);
\draw (0,0)--(1,2);
\draw (-.75,1.5)--(-.625 , 2);
\draw (-.25,1.5)--(-.375,2);
\node at (-.8,1) {$\nu_2$};
\node at (-.3,0) {$\nu_1$};
\node at (-.95,1.5) {$\nu_3$};
\node at (0,1.5) {$\nu_4$};
\end{tikzpicture}
}

\newcommand{\ttwoI}{
\begin{tikzpicture}[baseline=1.2cm]
\draw (-.5,1)--(-.5,.75);
\draw (-.5,1)--(-1,2);
\draw (-.5,1)--(0,2);
\draw (-.75,1.5)--(-.625 , 2);
\draw (-.25,1.5)--(-.375,2);
\draw (1,.75)--(1,2);
\node at (-.8,1) {$\nu_2$};
\node at (-.95,1.5) {$\nu_3$};
\node at (0,1.5) {$\nu_4$};
\end{tikzpicture}
}

For example, $$\text{ if } t = \ttwofone \text{ and } z=Y \text{ , then } f_z = \ttwoI \ .$$
We obtain that $n(z)=2, n(t)= 5, b(z) = 1, m(z) = 2$ and $\cV_1(f_z) = \{\nu_3,\nu_4\}$.
If $\tau\in\St(t,z)$, then necessarily $\tau(\nu_1)=0,\tau(\nu_2)=1$ and $\tau$ can take any values at $\nu_3$ and $\nu_4$.

Equality \eqref{eq:phitone} becomes
\begin{equation}\label{eq:phittwo}
\Phi(t)\de_e = \sum_{z\in E(t)} \al^{n(z) -1 } \beta^{b(z)} \sum_{\tau\in \St(t,z)} \al_{1,\tau}(f_z) \de_{W(t,\tau)}.
\end{equation}
\begin{notation}Write $\St(t,z)_+$ for the set of states $\tau$ satisfying that $z_\tau=z$ and such that there exists $\tau'\in\St(t')$ for which $\sigma W(t,\tau)=W(t',\tau').$\end{notation}
Proposition \ref{prop:sim} implies that:
\begin{equation}\label{eq:phittwo}
\phi(v) = \sum_{z\in E(t)} \al^{2n(z) -2 } \beta^{2b(z)} \sum_{\tau\in \St(t,z)_+} \al_{1,\tau}(f_z)^2.
\end{equation}

The following lemma provides a useful bound on the second part of the sum \eqref{eq:phittwo}.

\begin{lemma}\label{lem:S(tz)}
If $v=\dfrac{\sigma\circ t}{s}$ is a reduced fraction, then for any $z\in E(t)$, we have that 
\begin{equation}\label{eq:ztermone}
\sum_{\tau\in \St(t,z)_+} \al_{1,\tau}(f_z)^2 \leq (\al^4 + \beta^4)^{\dfrac{m(z)}{2}}.
\end{equation}
\end{lemma}

\begin{proof}
Fix $z\in E(t)$ and $\tau\in \St(t,z)_+$. Let $f=f_z$ be the unique forest satisfying that $t = f\circ z$.
It is easy to see that if $\nu\in \cV_1(f)$ and $\nu\sim \tilde\nu$ with $\tilde\nu\in \cV(t)$, then necessarily $\tilde \nu$ belongs to $\cV_1(f).$
We partition $\cV_1(f)$ as a union of equivalence classes $[\nu_1],\cdots , [\nu_k]$ w.r.t.~the relation $\sim$ where $\nu_1,\cdots,\nu_k$ is a set of representatives.
Let $m_j$ be the number of elements in the class $[\nu_j]$ and note that $m(z) = \sum_{j=1}^k m_k.$
We obtain that 
$$\al_{1,\tau}(f_z) = \al_{\tau,1}^{m_1}\cdots \al_{\tau,k}^{m_k}$$ where 
$$\al_{\tau,j} := \begin{cases}
\al \text{ if } \tau(\nu_j) =0 \\
\beta \text{ otherwise }
\end{cases}.$$
Therefore, 
$$\sum_{\tau\in \St(t,z)_+} \al_{1,\tau}(f_z)^2  = \sum_{\tau\in \St(t,z)_+}\al_{\tau,1}^{2m_1}\cdots \al_{\tau,k}^{2m_k}.$$
A state $\tau\in \St(t,z)_+$ is thus completely characterized by its values at $\nu_1,\cdots,\nu_k$. 
There are at most $2^k$ such states. 
Hence we obtain $$\sum_{\tau\in \St(t,z)_+} \al_{1,\tau}(f_z)^2\leq \sum_{\kappa} \kappa(1)^{m_1}\cdots\kappa(k)^{m_k}$$
where $\kappa$ runs over all maps from $\{1,\cdots,k\}$ to $\{\al^2,\beta^2\}.$
This sum is then equal to $\prod_{j=1}^k ( ( \al^2)^{m_j} + (\beta^2)^{m_j} )$ and thus
\begin{equation}\label{eq:prodone}
\sum_{\tau\in \St(t,z)_+} \al_{1,\tau}(f_z)^2 \leq \prod_{j=1}^k ( ( \al^2)^{m_j} + (\beta^2)^{m_j} ).
\end{equation}

Note that we have 
\begin{equation}\label{eq:prodtwo}
(\al^2)^m + (\beta^2)^m \leq (\al^4 + \beta^4)^{\dfrac{m}{2}} \text{ for any } m\geq 2.
\end{equation}
Indeed, assume that $\al\geq \beta$ and set $\rho := \dfrac{\beta^4}{\al^4}$ that is in $(0,1].$ 
Consider the function $$g(x) := (1+\rho)^x - (1+\rho^x)$$ for $x\geq 1.$
We have $$g'(x) = \log(1+\rho) (1+\rho)^x - \log(\rho) \rho^x$$ that is strictly positif for any $x\geq 1$ since $\log(\rho)\leq 0$ and $\log(1+\rho)>0.$
Therefore, $g$ is strictly increasing and thus $g(m/2) \geq g(1) = 0$ for any $m\geq 2.$
We obtain that $1 + \rho^{m/2} \leq (1+\rho)^{m/2}$ and thus Inequality \eqref{eq:prodtwo} by multiplying by $\al^{2m}$ for any $m\geq 2.$

By Proposition \ref{prop:sim} we have that $m_j\geq 2$ for any $1\leq j\leq k$.
Therefore, Inequalities \eqref{eq:prodone} and \eqref{eq:prodtwo} imply that
$$\sum_{\tau\in \St(t,z)_+} \al_{1,\tau}(f_z)^2  \leq \prod_{j=1}^k (\al^4 + \beta^4)^{\dfrac{m_j}{2}} = (\al^4 + \beta^4)^{\dfrac{m(z)}{2}}.$$
\end{proof}

Consider the map
$$h(n) := \dfrac{1}{2} \log_2( \dfrac{n}{2}),$$ 
where $\log_2$ is the logarithm in base $2$.
We now split rooted subtrees $z\in E(t)$ in two categories: the ones satisfying $m(z)>h(n)$ and the others.
Observe that 
\begin{equation}
\sum_{\substack{ z\in E(t)\\ m(z)>h(n)}} \al^{2n(z) -2 } \beta^{2b(z)} \sum_{\tau\in \St(t,z)_+} \al_{1,\tau}(f_z)^2 \leq \sum_{z\in E(t)} \al^{2n(z) -2 } \beta^{2b(z)} (\al^4 + \beta^4)^{ \dfrac{h(n)}{2} } = (\al^4 + \beta^4)^{ \dfrac{h(n)}{2} }.
\end{equation}
This term tends to zero as $n$ goes to infinity.
So we only need to consider the rest of rooted subtrees for which $m(z)\leq h(n)$.

\begin{lemma}\label{lem:sumS} 
We have the inequality $$\sum_{\substack{\tau\in \St(t) \\  m(z_\tau)\leq h(n)}} \al_\tau^2 \leq \al^{2 h(n)}.$$
\end{lemma}

\begin{proof}
We start by proving that there exists a subset of vertices $A\subset \cV(t)$ having $h(n)$ elements that is contained in the vertex set of any rooted subtree $z\in E(t)$ satisfying that $m(z)\leq h(n)$, i.e.~ $$| \bigcap_{\substack{z\in E(t)\\ m(z)\leq h(n)}} \cV(z) | \geq h(n).$$

Recall that $t$ is a tree with $n$ leaves and thus has $n-1$ trivalent vertices.
Consider the longest geodesic path $c$ inside $t$ starting from the root and ending at one leaf. 
We claim that the length $|c|$ of this path is larger than $2h(n)+1.$
Assume by contradiction that any path in $t$ has length less than $2h(n).$
This implies that $t$ is a rooted subtree of the full rooted binary tree having $2^{2h(n)}$ leaves all at distance $2h(n)$ from the root. 
This tree has $2^{2h(n)}-1$ vertices that is $2^{\log_2(n/2)}-1 = n/2-1$.
Since $t$ has $n-1$ vertices we obtain a contradiction.

Therefore, there exists a path $c\in \Path(t)$ of length larger than $2h(n)+1.$ The path $c$ contains at least $2h(n)$ trivalent vertices of $t$. 
Consider a rooted subtree $z\in E(t)$ such that $m(z)\leq h(n).$ 
There are at most $h(n)+1$ vertices of $c$ that are not inside $z$. Those vertices are necessarily the one at the end of $c$ that are the $h(n)+1$ last one. Therefore, $\cV(z)$ contains at least the $h(n)$ first vertices of $c$. This proves that there is a subset $A\subset\cV(t)$ of $h(n)$ elements contained in every rooted subtree $z\in E(t)$ for which $m(z)\leq h(n).$
Therefore, if $\tau$ is a state on $t$ satisfying that $m(z_\tau)\leq h(n)$, then $\tau(\nu)=0$ for any $\nu\in A$.
Therefore, $$\sum_{\substack{\tau\in \St(t)  \\ m(z_\tau)\leq h(n)}} \al_\tau^2 \leq \al^{2|A|} \sum_{\gamma} \al_\ga^2,$$
where $\ga$ runs over every maps from $\cV(t)\setminus A \to \{0,1\}$ and where $\al_\ga=\al^{|\ga^{-1}(0)|}\beta^{|\ga^{-1}(1)|}$. But $\sum_\ga \al_\ga^2 = 1$ and thus 
$$\sum_{\substack{\tau\in \St(t)  \\ m(z_\tau)\leq h(n)}} \al_\tau^2 \leq \al^{2h(n)}.$$
\end{proof}

\subsection{End of the proof.}

For $v=\dfrac{\sigma\circ t}{s}$ a reduced fraction with trees having $n$ leaves we have the following:

\begin{align*}
\phi(v) & = \sum_{z\in E(t)} \al^{2n(z) -2 } \beta^{2b(z)} \sum_{\tau\in \St(t,z)_+} \al_{1,\tau}(f_z)^2 \text{ by } \eqref{eq:phittwo}\\
& \leq \sum_{z\in E(t)} \al^{2n(z) -2 } \beta^{2b(z)} (\al^4 + \beta^4)^{\dfrac{m(z)}{2}} \text{ by Lemma \ref{lem:S(tz)} }\\
& \leq \sum_{\substack{z\in E(t) \\ m(z)>h(n)}} \al^{2n(z) -2 } \beta^{2b(z)} (\al^4 + \beta^4)^{\dfrac{h(n)}{2}} + \sum_{\substack{ z\in E(t)\\ m(z) \leq h(n) } }\sum_{\tau\in \St(t,z)} \al_\tau^2\\
&\leq \left(\sum_{\substack{z\in E(t)\\ m(z)>h(n)}} \al^{2n(z) -2 } \beta^{2b(z)}\right) (\al^4 + \beta^4)^{\dfrac{h(n)}{2}} + \al^{2 h(n)} \text{ by Lemma \ref{lem:sumS} }\\
&\leq \left( \sum_{z\in E(t) } \al^{2n(z) -2 } \beta^{2b(z)} \right) (\al^4 + \beta^4)^{\dfrac{h(n)}{2}} + \al^{2 h(n)} \\
&\leq (\al^4 + \beta^4)^{\dfrac{h(n)}{2}} + \al^{2 h(n)} \text{ since } \sum_{z\in E(t)} \al^{2n(z) -2 } \beta^{2b(z)} =1.\\
\end{align*}
Since $\lim_{n\to\infty} h(n)=\infty$ and $0<\al , \al^4+\beta^4 <1$, we obtain that $\lim_{n\to\infty}\sup_{V\setminus V_n}|\phi(v)|=0$ where $V_n$ is the subset of $V$ of elements that can be written as a fraction of symmetric trees with less than $n-1$ leaves.
Since $(V_n)_n$ is an increasing sequence of finite subsets of $V$ whose union is equal to $V$ we obtain that $\phi$ vanishes at infinity.

\begin{remark}\label{rem:Farley}
We have proven that for any $0<\al<1$ the map $\phi_\al:V\to\C$ is a positive definite function that vanishes at infinity. Moreover, $\lim_{\al\to 1}\phi_\al(v)=1$ for any $v\in V$ implying that $V$ has the Haagerup property.
This theorem was first proved by Farley where he defined a proper cocycle on $V$ with value in a Hilbert space \cite{Farley03-H}.
Using Schoenberg Theorem applied to the square of the norm of this cocycle we obtain a one parameter family of positive definite maps $f_\al:V\to\C, 0<\al<1$ satisfying that $f_\al(v)= \al^{2n(v)-2}$ where $n(v)$ is the minimum number of leaves for which $v$ is described by a fraction of symmetric trees with $n(v)$ leaves.
In \cite{Brothier-Jones19}, Jones and the author constructed a family of positive definite maps on $V$ that coincide with the maps of Farley when restricted to Thompson's group $T$, see \cite[Remark 1]{Brothier-Jones19}, but do not vanishes at infinity on the group $V$.
A similar observation shows that the restriction to $T$ of our maps $\phi_\al$ coincide with the maps of Farley. 
However, those three families of maps no longer coincide on the whole group $V$.
\end{remark}

\section{A class of wreath products with the Haagerup property}\label{sec:wreathprod}
Following the preliminary section we consider a group $\Ga$, an injective morphism $S:\Ga\to \Ga\oplus\Ga$, the associated monoidal functor 
$$\Xi:\cF\to\Gr,\ \Xi(1)=\Ga,\ \Xi(Y)=S,$$ 
and the associated category $\cC_\Xi=\cC$.
Write $G_\cC$ for the group of fractions of the category $\cC$ (at the object $1$). 

\subsection{Constructions of unitary representations}

Given a representation of $\Ga$ and an isometry $R:\fH\to\fH\otimes \fH$ we want to construct a representation of the larger group $G_\cC.$
To do this we will define a monoidal functor $\Psi:\cC_\Xi\to\Hilb$ and then use Jones' technology.
We start by explaining how to build such a functor.

\begin{proposition}\label{prop:functor}
There is a one to one correspondance between monoidal functors $\Psi:\cC_\Xi\to \Hilb$ and pairs $(\rho,R)$ satisfying the properties:
\begin{enumerate}
\item $\rho:\Ga\to \cU(\fH)$ is a unitary representation;
\item $R:\fH\to \fH\ot\fH$ is an isometry;
\item $R\circ \rho(g) = (\rho\ot\rho)(S(g)) \circ R, \ \forall g\in \Ga$.
\end{enumerate}
The correspondance is given by $$\Psi\mapsto (\rho_\Psi , \Psi(Y))$$ where $\rho_\Psi(g) : = \Psi(g)$ for all $g\in\Ga.$
\end{proposition}

\begin{proof}
Consider a monoidal functor $\Psi$ and the associated couple $(\rho,R)$. 
The two first properties come from the fact that morphisms of $\Hilb$ are linear isometries.
The third property results from the computation of $\Psi(Y\circ g)$ and the equality $Y\circ g=S(g)\circ Y$ inside the category $\cC$ for all $g\in\Ga.$
Since any morphism of $\cC_\Xi$ is the composition of tensor products of $g\in\Ga$, the tree $Y$, and some permutations we have that those properties completely characterized $\Psi$ and are sufficient.
\end{proof}
Note that a functor $\Psi$ as above satisfies the equality $$\Psi(f)\circ \rho^{\ot n}(g) = \rho^{\ot m} (\Xi(f)(g)) \circ \Psi(f), \ \forall f\in\cF(n,m), g\in \Ga^n.$$

{\bf Assumption.} 
From now one we assume that $S(g)=(g,e)$ and thus the group of fractions $G_\cC$ is isomorphic to $\oplus_{\Q_2}\Ga\rtimes V$ by Proposition \ref{prop:Bernoulli}. 
We will build specific coefficients for $G_\cC$ using Jones' representations arising from Proposition \ref{prop:functor}.

\subsection{Constructions of matrix coefficients}\label{sec:coef}
From any coefficient of $\Ga$ and coefficient $\phi_\al$ of $V$ (as constructed in Section \ref{sec:coefV}) we build a coefficient of the larger group $G_\cC\simeq \oplus_{\Q_2}\Ga\rtimes V$.

{\bf Positive definite maps on the group $\Ga$.}
Let $\phi_\Ga:\Ga\to \C$ be a positive definite function on $\Ga.$
There exists a unitary representation $(\kappa_0,\fK_0)$ and a unit vector $\xi\in\fK_0$ such that 
$$\phi_\Ga(g) = \langle \xi , \kappa_0(g)\xi\rangle \text{ for any } g\in \Ga.$$
For technical purpose we consider the infinite tensor product of the representation $\kappa_0$. 
In order to take an infinite tensor product we must first add a vector on which the group acts trivially.
Define $\fK:=\fK_0\oplus \C\Omega$ where $\Omega$ is a unit vector and extend the unitary representation $\kappa_0$ as follows:
$$\kappa(g)(\eta\oplus \mu \Omega) = (\kappa_0(g)\eta)\oplus \mu \Omega \text{ for any } g\in \Ga, \eta\in \fK_0, \mu\in \C.$$
Hence, $\kappa$ is the direct sum of $\kappa_0$ and the trivial representation $1_\Ga.$
Let $\fK^\infty$ be the infinite tensor product $\ot_{k\geq 1}(\fK,\Omega)$ with base vector $\Omega.$
In other words $\fK^\infty$ is the completion of the directed system of Hilbert spaces $(\fK^{\ot n}, n\geq 1)$ with inclusion maps 
$$\iota_n^{n+p}:\fK^{\ot n} \to \fK^{\ot n+p} , \eta\mapsto \eta\ot \Omega^{\ot p} \text{ for } n,p\geq 1.$$
For any $g\in \Ga$ we define the following map:
$$\kappa^\infty(g)(\ot_{k\geq 1} \eta_k) = \ot_{k\geq 1} \kappa(g)\eta_k$$
for an elementary tensor $\ot_{k\geq 1} \eta_k$ such that $\eta_k=\Omega$ for $k$ large enough.
This formula defines for any $n$ a unitary representation of $\Ga$ on $\fK^{\ot n}$. This family of representations is compatible with the directed system of Hilbert spaces and thus defines a unitary representation $$\kappa^\infty:\Ga\to \cU(\fK^\infty).$$

{\bf Isometries for the Thompson group $V$.}
Consider $0\leq \al\leq 1$ and the map $R_\al:\fH\to\fH\ot\fH$ defines in Section \ref{sec:coefV}.
Hence, $\fH=\ell^2(M)$ where $M$ is the free monoid in four generators $a,b,c,d.$
Moreover, recall that we write $\beta$ for $\sqrt{\al^2-1}$ and we have
\begin{align*}
R_\al:&\fH\to \fH\ot \fH\\
&\de_e \mapsto \al \de_{e,e} + \beta \de_{c,d}\\
& \de_x\mapsto \al\de_{xa,xb} + \beta\de_{xc,xd} \text{ for all } x\in M, x\neq e.
\end{align*}

{\bf Mixing representations of $\Ga$ with isometries.}
We can now build a monoidal functor from $\cC$ to $\Hilb$ and a matrix coefficient for its group of fractions $G_\cC.$
Define the Hilbert space 
$$\fL:=\fK^\infty\ot \ell^2(M)$$ and the map:
$$R=R_{\phi_\Ga,\al}:\fL\to\fL\ot \fL$$ 
as follows:
\begin{align*}
R(\eta\ot\de_e) & = \al (\eta\ot\de_e)\ot(\xi\ot\de_{e}) + \beta (\eta\ot\de_c)\ot(\xi\ot\de_{d})\\
R(\eta\ot\de_x) & = \al (\eta\ot\de_{xa})\ot ( \xi^{\ot |x|+1}\ot\de_{xb}) + \beta (\eta\ot\de_{xc})\ot(\xi^{\ot |x|+1}\ot\de_{xd}) \text{ if } x\neq e.
\end{align*}
Note that up to flipping tensors we have the formula 
$$R(\eta\ot\de_x) = (\eta\ot\xi^{\ot |x|+1}) \ot R_\al(\de_x) \text{ for } x\in M, \eta\in\fK^\infty.$$
Observe that in the formula we have $\xi$ elevated to certain tensor powers. This will permit to have matrix coefficients tending quickly to $0$ at infinity. This is the reason why we consider $\fK^\infty$ rather than $\fK_0$.
Define the unitary representation $$\rho:=\kappa^\infty\ot 1:\Ga\to\cU(\fL)$$ such that 
$$\rho(g)(\eta\ot \zeta) = \kappa^\infty(g)(\eta)\ot \zeta$$
for any $g\in \Ga, \eta\in\fK^\infty, \zeta\in \ell^2(M).$

The following proposition is straightforward:

\begin{proposition}
The pair $(\rho,R)$ verifies the assumptions of Proposition \ref{prop:functor}.
Hence, there exists a unique monoidal functor $\Psi=\Psi_{\phi_\Ga,\al}:\cC_\Xi\to\Hilb$ satisfying that 
$$\Psi(1)=\fL, \Psi(Y)= R_{\phi_\Ga,\al} \text{ and } \Psi(g) = \rho(g) \text{ for any } g\in\Ga.$$
\end{proposition}

Let us apply the Jones construction to the functor $\Psi=\Psi_{\phi_\Ga,\al}$ of the proposition.
We obtain a Hilbert space $\scrL_{\phi_\Ga,\al}$ and a unitary representation of the group of fractions of $\cC=\cC_\Xi$ that is:
$\pi_{\phi_\Ga,\al}:G_\cC\to\cU(\scrL_{\phi_\Ga,\al}).$
We now build a coefficient for $G_\cC.$
Consider the unit vector $\xi\ot\de_e\in\fL$ view as a vector of the larger Hilbert space $\scrL=\scrL_{\phi_\Ga,\al}$ and set
$$\varphi_{\phi_\Ga,\al}:G_\cC\to\C, v_g\mapsto \langle \pi_{\phi_\Ga,\al}(v_g)\xi\ot\de_e,\xi\ot\de_e\rangle.$$

\begin{lemma}\label{lem:pj}
Let $t$ be a tree and $\tau$ a state on $t$. 
Decompose $t$ as $f_\tau\circ z_\tau$ (see Notation \ref{not:ztau}).
Consider the geodesic path in $f_\tau$ starting at a root and ending at the $j$-th leaf and its subpath with same start but ending at the last right-edge of the path. 
If this subpath is empty (has length zero), we set $L_j(\tau,t)=L_j(\tau)=1$.
Otherwise, we set $L_j(\tau,t)=L_j(\tau)$ the length of this subpath. 
We have that
$$\Psi(t)(\xi\ot \de_e) = \sum_{\tau\in \St(t)} \al_\tau \xi^{\ot L(\tau)} \ot \de_{W(t,\tau)}$$
(up to the identification $\fL^{\ot n} \simeq (\fK^\infty)^{\ot n} \ot \ell^2(M^n)$)
where $$\xi^{\ot L(\tau)} : = \xi^{\ot L_1(\tau)}\ot\cdots\ot\xi^{\ot L_n(\tau)} \in (\fK^\infty)^{\ot n}$$
and where $W(t,\tau)$ is the list of words in the free monoid $M$ defined in Section \ref{sec:state}.
\end{lemma}

The proof follows from an easy induction on the number of vertices of $f_\tau.$
Rather than proving it we illustrate the formula on one example.
Consider the following tree:
$$\ttwofone.$$
Define the state $\tau$ such that $\tau(\nu_1)=0, \tau(\nu_2)=1, \tau(\nu_3)=0,\tau(\nu_4)=1.$
We then have that $z_\tau = Y$ and $f_\tau = t_2\ot I$ where $t_2$ is the full rooted binary tree with $4$ leaves all at distance $2$ from the root.
Since $\tau$ takes the value $0$ twice and the value $1$ twice we obtain that $\al_\tau=\al^2\beta^2.$
Following each geodesic path from the root to the $j$-th leaf and considering the state $\tau$ at each vertex we obtain that 
$$W(t,\tau)=(ca,cb,dc,dd,e).$$
The geodesic path in $f_\tau$ from a root to the first leaf is a succession of two left-edges. 
So the subpath ending with a right-edge is trivial and thus has length zero. 
We then put $L_1(\tau)=1.$
The second subpath is a left-edge followed by a right-edge and thus $L_2(\tau)=2$.
Looking at the other leaves we obtain that $L_1(\tau)=1,L_2(\tau)=2,L_3(\tau)=1,L_4(\tau)=2,L_5(\tau)=1$. 
Applying the formula of the proposition we get that the $\tau$-component of $\Phi(t)(\xi\ot\de_e)$ is equal to 
$$\al^2\beta^2 (\xi \ot \de_{ca}) \ot (\xi\ot\xi\ot \de_{cb}) \ot (\xi\ot\de_{dc}) \ot (\xi\ot\xi\ot \de_{dd})\ot (\xi\ot\de_e).$$
Another way to compute $L_j(\tau)$ is to look at the longest subword of $W(t,\tau)_j$ starting at the first letter and ending at the last $b$ or $d$-letter. If this words is trivial (there are no $b$ or $d$-letter) we put $L_j(\tau)=1$. Otherwise, $L_j(\tau)$ is the length of this word. 

\subsection{Matrix coefficients vanishing at infinity and the Haagerup property}

The next proposition proves that a large class of matrix coefficients of $G_\cC$ vanish at infinity.
This is the key technical result for proving that wreath products have the Haagerup property.

\begin{proposition}\label{prop:Czero}
Consider a discrete group $\Ga$ and a positive definite map $\phi_\Ga:\Ga\to\C$ satisfying that there exists $0\leq c<1$ such that $|\phi_\Ga(g)|\leq c$ for any $g\neq e$ and that vanishes at infinity.
If $0<\al<1$ and $\varphi=\varphi_{\phi_\Ga,\al}$ is the coefficient built in Section \ref{sec:coef}, then it vanishes at infinity.
\end{proposition}

\begin{proof}
Consider trees $t,t'$ with $n$ leaves, a permutation $\sigma\in S_n$ and $g=(g_1,\cdots,g_n)\in \Ga^n$.
Write $v=\dfrac{\sigma t}{t'}\in V$ and $v_g=\dfrac{g\sigma t}{t'}\in G_\cC$.
Recall that any element of $G_\cC$ can be written in that way.
Fix $0<\varep<1$ and assume that $|\varphi(v_g)|\geq \varep.$
Let us show that there are only finitely many such $v_g$.

By definition of the coefficients we have that $|\varphi(v_g)|\leq \prod_{j=1}^n |\phi_\Ga(g_j)|.$
Since the map $\phi_\Ga:\Ga\to\C$ vanishes at infinity and $|\varphi(v_g)|\geq \varep$ we obtain that there exists a finite subset $Z\subset \Ga$ such that $g\in Z^n.$

Observe that $|\varphi(v_g)|\leq |\phi_\al(\dfrac{\sigma t}{t'})|$ where $\phi_\al:V\to\C$ is the coefficient built in Section \ref{sec:coefV}.
We proved in Section \ref{sec:VH} that $\phi_\al$ vanishes at infinity. 
Therefore, we may write $\dfrac{\sigma t}{t'}$ as a fraction with few leaves.
Hence, there exists a fixed $N\geq 1$ depending solely on $\varep$ such that 
\begin{equation}\label{eq:frac}\dfrac{\sigma t}{t'}=\dfrac{\theta t_N}{s}\end{equation} 
for some tree $s$ and permutation $\theta$ and where $t_N$ denotes the full rooted binary tree with $2^N$ leaves all at distance $N$ from the root

The next claim will show that the fraction $\dfrac{g\sigma t}{t'}$ can be reduced as a fraction $\dfrac{g'\theta t_{N'}}{s}$ for some $N'\geq 1$ that only depends on $N$ (and thus only depends on $\varep$).
To do this we need to show that if $g_j$ is nontrivial, then the geodesic path inside $t$ ending at the $j$-th leaf is mainly a long path with only left-edges. 
Define $P_j$ to be the geodesic path from the root of the tree $t$ to the $j$-th leaf of $t$ and write $P_j^R$ its subpath starting at the root and ending at the last right-edge of $P_j$.

{\bf Claim:} We have the inequality 
\begin{equation}\label{eq:pathPj}|\varphi(v_g)| \leq   (|P_j^R|+1)\max(\al^2,|\phi(g_j)|)^{|P_j^R|}\end{equation}
for any $1\leq j\leq n.$

{\bf Proof of the claim:}
Lemma \ref{lem:pj} states that 
$$\Phi(t)\xi\ot \de_e = \sum_{\tau\in \St(t)} \al_\tau \xi^{\ot L(\tau)}\ot  \de_{W(t,\tau)}.$$
Therefore, 
\begin{align*}
\varphi(v_g) & = \langle\Phi(t')\xi\ot\de_e , \Phi(g\sigma)\Phi(t)\xi\ot\de_e\rangle\\
& = \sum_{\tau\in \St(t)}\sum_{\tau'\in \St(t')} \langle \al_{\tau'} \xi^{\ot L(\tau')}\ot  \de_{W(t',\tau')} , \Phi(g\sigma)\al_\tau \xi^{\ot L(\tau)}\ot  \de_{W(t,\tau)}\rangle\\
& = \sum_{\tau\in \St(t)}\sum_{\tau'\in \St(t')} \langle \al_{\tau'} \xi^{\ot L(\tau')}\ot  \de_{W(t',\tau')} , \al_\tau (\kappa(g)\xi)^{\ot \sigma L(\tau)}\ot  \de_{\sigma W(t,\tau)}\rangle\\
& = \sum_{\tau\in \St(t)}\sum_{\tau'\in \St(t')} \al_{\tau'}\al_{\tau} \prod_{i=1}^n \phi_\Ga(g_i)^{L_{i}(\tau)} \langle  \de_{W(t',\tau')}, \de_{\sigma W(t,\tau)}\rangle.
\end{align*}
By Propostion \ref{prop:sim}, we have that given a state $\tau\in \St(t)$ there are at most one $\tau'\in \St(t')$ such that $W(t',\tau') = \sigma W(t,\tau)$ and in that case $\al_\tau=\al_{\tau'}.$ 
This implies that 
$$|\varphi(v_g)|  \leq \sum_{\tau\in \St(t)} \al_{\tau}^2 \prod_{i=1}^n |\phi_\Ga(g_i)|^{ L_i(\tau) }.$$

Fix $1\leq j\leq n$ and consider the set of vertices of the path $P_j^R$ that we denote from bottom to top by $\nu_1,\nu_2,\cdots,\nu_q$.
Our convention is that the last vertex $\nu_q$ is the source of the last edge of $P_j^R$ and thus $|P_j^R|= q.$
Define $$S_k = \begin{cases} \{\tau\in\St(t): \tau(\nu_1)=1\} \text{ if } k=0;\\
\left\{ \tau\in\St(t): \tau(\nu_1)=\cdots=\tau(\nu_k)=0, \tau(\nu_{k+1})=1\right\} \text{ if } 1\leq k \leq q-1;\\
\{\tau\in\St(t): \tau(\nu_1)=\cdots=\tau(\nu_q)=0\} \text{ if } k=q.
\end{cases}$$
Observe that 
$$
\sum_{\tau\in S_k} \al_\tau^2 =
\begin{cases} 
\al^{2k}\beta^2 \text{ if } 0\leq k\leq q-1 \\
\al^{2q} \text{ if } k=q
\end{cases}.
$$
Moreover, if $\tau\in S_k$ for $0\leq k\leq q$, then $L_j(\tau)=q-k.$
Therefore,
\begin{align*}
|\varphi(v_g)| & \leq \sum_{\tau\in \St(t)} \al_{\tau}^2 \prod_{i=1}^n |\phi_\Ga(g_i)|^{ L_i(\tau) }\\
& \leq \sum_{\tau\in \St(t)} \al_{\tau}^2 |\phi_\Ga(g_j)|^{ L_j(\tau) }
 = \sum_{k=0}^{q} \sum_{\tau \in S_k} \al_{\tau}^2 |\phi_\Ga(g_j)|^{ L_j(\tau) }\\
& = \sum_{k=0}^{q-1} \al^{2k}\beta^2 |\phi_\Ga(g_j)|^{q-k } + \al^{2q} |\phi_\Ga(g_j)|\\
&\leq \sum_{k=0}^{q-1} \max(\al^2,|\phi_\Ga(g_j)|)^{q}  + \max(\al^2,|\phi_\Ga(g_j)|)^{q+1}\\
& \leq (q+1)\max(\al^2,|\phi_\Ga(g_j)|)^{q}.
\end{align*}
This proves the claim.

We now explain how to reduce our fraction $\dfrac{g\sigma t}{s}$.

{\bf Claim:} There exists $Q\geq 1$ such that $|P_j^R| \leq Q$ for any $j\in J$ where $J:=\{j:\ g_j\neq e\}$ is the support of $g$.

If $J$ is empty, then we can take $Q=1$.
Assume $J$ is nonempty and take $j\in J$.
By assumption we have that $|\phi(g_j)|<c$ for a fixed constant $0<c<1.$
Moreover, $0<\al<1.$
This implies that the quantity $(P+1)\max(\al^2, |\phi(g_j)|)^P$ tends to zero in $P$.
Therefore, by the preceding claim we deduce that there exists $Q\geq 1$ such that $|P_j^R| \leq Q$ for any $j\in J.$
This proves the claim.

From the claim we deduce that the geodesic path $P_j$ from the root of $t$ to its $j$-th leaf with $j\in J$ is the concatenation of a first path $P_j^R$ of length less than $Q$ ending with a right-edge and a second path which consists on a succession of left-edges.
Using the rules of composition of morphisms in the category $\cC_\Xi$ we can write the composition $g\circ\sigma \circ t$ in a different fashion as follows.
First observe that $g\circ \sigma = \sigma \circ g_\sigma$ where $g_\sigma\in \Ga^n$ whose $i$-th component is $g_{\sigma(i)}.$
Second we make the group elements go down in the tree using the relation $(x,e)Y = Y x$ for $x\in \Ga.$
We apply this relation to any nontrivial group element $g_j, j\in J$ along the second part of the path $P_j$ that is a succession of left-edges.
We obtain that $g\circ \sigma \circ t = f \circ \sigma' \circ g' \circ t'$ for some $f,\sigma',g',t'$ satisfying that $\sigma t = f\sigma' t'$ and such that $g'\in Z^{n'}$ for some $n'\leq n.$
We can choose $t'$ for which every leaf is at most at distance $Q$ from the root and thus can be seen as rooted subtree of the complete binary tree $t_Q$ that has $2^Q$ leaves all of them at distance $Q$ from the root.
We obtain that $$v_g = \dfrac{ f\sigma' g' t_Q}{f' t''}.$$

Using \eqref{eq:frac}, we obtain that $v_g$ can be reduced as a fraction $\dfrac{g'\sigma' t_U}{t''}$ where $U=\max(N,Q)$ and $g'\in Z^{n'}$ where $n'=2^U.$
Since $Z$ is finite and $U$ is fixed (and only depends on $\varep$) there are only finitely many such fractions implying that $\varphi$ vanishes at infinity.
\end{proof}

We are now able to prove one of the main theorems of this article.

\begin{theorem}\label{theo:Htwo}
If $\Ga$ is a discrete group with the Haagerup property, then so does the wreath product $\oplus_{\Q_2} \Ga\rtimes V$.
\end{theorem}

\begin{proof}
Fix a discrete group $\Ga$ with the Haagerup property.
By Proposition \ref{prop:Bernoulli} the wreath product $\oplus_{\Q_2} \Ga\rtimes V$ is isomorphic to the group of fractions $G_\cC$ and thus it is sufficient to prove that this later group has the Haagerup property.
Consider a finite subset $X\subset G_\cC$ and $0<\varep<1.$
Since $X$ is finite there exists $n$ and a finite subset $Z\subset \Ga$ such that $X\subset X_n$ where $X_n$ is the set of fractions $v_g:=\dfrac{g\sigma t}{s}$ where $t,s$ are trees with $n$ leaves, $g=(g_1,\cdots,g_n)\in Z^n$ and $\sigma\in S_n.$
Fix $\varep'>0$ the unique positive number satisfying that 
$$(1-\varep')^{2n+n^2}=1-\varep.$$
Since $\Ga$ has the Haagerup property there exists a positive definite map $\phi_\Ga:\Ga\to\C$ vanishing at infinity satisfying that $|\phi_\Ga(x)|> 1-\varep'$ for any $x\in Z.$ 

Since $\Ga$ is discrete we can further assume that there exists $0<c<1$ satisfying that $|\phi_\Ga(x)|\leq c$ for any $x\in \Ga,x\neq e$.
Indeed, if $\phi_\Ga(g)=\langle \xi,\kappa(g)\xi\rangle$ for some representation $(\kappa,\fK)$ we consider $(\kappa\oplus\la_\Ga,\fK\oplus\ell^2(\Ga))$ where $\la_\Ga$ is the left-regular representation of the discrete group $\Ga$. 
Given any angle $\theta$ we set 
$$\eta:=\cos(\theta) \xi \oplus \sin(\theta)\de_e$$ and define the coefficient $$\psi_\theta(g)=\langle\eta,(\kappa\oplus\la)(g)\eta\rangle, g\in \Ga.$$
Note that $\eta$ is a unit vector and that 
$$\psi_\theta(g)=\begin{cases} \cos(\theta)^2 \phi_\Ga(g) \text{ if } g\neq e\\ 1 \text{ if } g=e\end{cases}.$$
We then replace $\phi_\Ga$ by $\psi_\theta$ for $\theta$ sufficiently small.

Consider the map $\phi_\al:V\to\C$ of Section \ref{sec:coefV} with parameter $\al=1-\varep'$ and denote by $\varphi=\varphi_{\phi_\Ga,\al}$ the associated coefficient of $G_\cC.$
By Proposition \ref{prop:Czero}, the map $\varphi$ vanishes at infinity on $G_\cC$.
Consider $v_g\in X$ and observe that 
$$|\varphi(v_g)|\geq \al^{2n-2} \prod_{j=1}^n |\phi_\Ga(g_j)|^n\geq (1-\varep')^{2n-2} (1-\varep')^{n^2}\geq (1-\varep')^{2n+n^2}=1-\varep.$$
Hence, for any finite subset $X\subset G_\cC$ and $0<\varep<1$ there exists a positive definite map $\varphi:G_\cC\to\C$ vanishing at infinity and satisfying that $|\varphi(v)|\geq 1-\varep$ for any $v\in X.$
This implies that $G_\cC$ has the Haagerup property.
\end{proof}

\subsection{Haagerup property for twisted wreath products}

In this section we fix a group $\Ga$ and an automorphism of it $\theta\in\Aut(\Ga)$.
Recall from Section \ref{sec:WP} that this defines a category $\cC=\cC_{\Ga,\theta}$ where morphisms are forests with leaves labelled by elements of $\Ga$ and by permutations satisfying the relation
$$Y\circ g = (\theta(g),e)\circ Y.$$
Moreover, the group of fractions of $\cC$ is isomorphic to the twisted wreath product $\Ga\wr^\theta_{\Q_2}V.$
By adapting the proof of Theorem \ref{theo:Htwo} we obtain the following result.

\begin{theorem}\label{theo:twisted-WP-Haagerup}
If $\Ga$ is a group with the Haagerup property and $\theta\in\Aut(\Ga)$ is an automorphism, then the twisted wreath product $\Ga\wr_{\Q_2}^\theta V$ has the Haagerup property.
\end{theorem}

\begin{proof}
Fix a group $\Ga$ with the Haagerup property and an automorphism $\theta$ of it.
Denote by $G$ the twisted wreath product $\Ga\wr_{\Q_2}^\theta V$ that we identify with the group of fractions of the category $\cC_{\Ga,\theta}.$
We mainly follow the construction explained in Section \ref{sec:coef} and keep similar notations.
We choose a positive definite function $\phi_\Ga:\Ga\to\C$ realized as $\phi_\Ga(g)=\langle \xi, \kappa_0(g)\xi\rangle$ and put $\fK=\fK_0\oplus\C\Omega$.
Consider $\fK^\infty:=\ot_{n\geq 0} (\fK,\Omega)$ and the associated representation of $\Ga$ denoted by $\kappa^\infty.$

Now, we modify the construction by considering the automorphism $\theta$.
We define $\fK_\theta:=\oplus_{n\in\Z} \fK^\infty$ the infinite direct Hilbert space sum of $\fK^\infty$ over the set $\Z$ and the representation
$$\kappa_\theta:=\oplus_{n\in\Z}( \kappa^\infty\circ\theta^{-n}).$$
Consider the operator $\shift:\fK_\theta\to\fK_\theta$ defined as 
$$\shift(\oplus_{n\in\Z} \eta_n):= \oplus_{n\in\Z} \eta_{n-1}.$$
This is a unitary satisfying
\begin{equation}\label{eq:shift}\kappa_\theta(\theta(g))\circ \shift = \shift \circ \kappa_\theta(g) \text{ for any } g\in \Ga.\end{equation}
We set $\fL:=\fK_\theta\ot\ell^2(M)$ and the unitary representation $\rho_\theta:=\kappa_\theta\ot 1$ similarly than before.
We now define our $R$-map. To do this we need to replace our favourite vector $\xi$ by one that is almost invariant by the shift operator.
Given any vector $\eta\in \fK$ and $n\geq 1$ we put:
$$\eta_n:=  \frac{1}{\sqrt{2n+1}} \oplus_{k\in\Z} \eta \ \chi_{[-n,n]}(k)\in \fK_\theta$$ 
where $\chi_{[-n,n]}$ is the characteristic function of $\{k\in\Z:\ |k|\leq n\}.$
Note that if $\eta$ is a unit vector, then $\eta_n$ is again a unit vector satisfying $\langle\shift(\eta_n),\eta_n\rangle = \frac{2n}{2n+1}.$
We will then consider vectors like $\xi_n$ and $\xi_n^{\ot |x|+1}$ in $\fK_\theta$.
The new $R$-map from $\fL$ to $\fL\ot\fL$ is the following:
\begin{align*}
R(\eta\ot\de_e) & = \al (\shift(\eta)\ot\de_e)\ot(\xi_n\ot\de_{e}) + \beta (\shift(\eta)\ot\de_c)\ot(\xi_n\ot\de_{d})\\
R(\eta\ot\de_x) & = \al (\shift(\eta)\ot\de_{xa})\ot ( \xi_n^{\ot |x|+1}\ot\de_{xb}) + \beta (\shift(\eta)\ot\de_{xc})\ot(\xi_n^{\ot |x|+1}\ot\de_{xd}),
\end{align*}
for $x\neq e$.
It is the same formula than in the untwisted case except that $\eta,\xi,\xi^{\ot |x|+1}$ are replaced by $\shift(\eta),\xi_n,\xi_n^{\ot |x| +1},$ respectively.
By reordering the tensors we obtain the following short formula:
$$R(\eta\ot \de_x) = (\shift(\eta)\ot \xi_n^{|x|+1}) \ot R_\al(\de_x).$$
One can check that $(R,\rho)$ defines a monoidal functor from $\cF$ to $\Hilb$ and thus a Jones' representation $\pi:G\to\cU(\mathscr L).$
We consider the positive definite function: $$\varphi:=\varphi_{n,\al,\phi_\Ga}(\ga):=\langle \pi(\ga)\xi_n\ot\de_e,\xi_n\ot\de_e\rangle \text{ for any } \ga\in G.$$
A similar proof can be applied by considering $\phi_\Ga$ as in the proof of Theorem \ref{theo:Htwo}, letting $\al$ tending to one and $n$ to infinity. 
We then obtain a net of positive definite functions $\varphi_{n,\al,\phi_\Ga}$ vanishing at infinity and tending to one thus proving that the group of fraction $G$ has the Haagerup property.
\end{proof}

The following proposition shows that we have many new examples of wreath products with the Haagerup property; indeed the wreath product $\Ga\wr_{\Q_2}^\theta W$ with $W$ being $F,T,$ or $V$ remembers the group $\Ga$ and the automorphism $\theta$.
It was proven in \cite[Theorem 4.12]{Brothier22} for the $V$-case. The untwisted version of it has been proven for the $F$ and $T$-cases in \cite[Theorem 4.1]{Brothier22-HPM} and can easily be extended to the twisted case. We leave the proof of this extension to the reader.

\begin{proposition}
Consider two pairs of groups with an automorphism $(\Ga,\theta)$ and $(\ti\Ga,\ti\theta)$. 
Let $G,\ti G$ be the associated twisted wreath products $\Ga\wr_{\Q_2}^\theta V$ and $\ti\Ga\wr_{\Q_2}^{\ti\theta} V$.
We have that $G\simeq \ti G$ if and only if there exists an isomorphism $\beta:\Ga\to \ti \Ga$ and $h\in \ti\Ga$ satisfying $\ti\theta = \ad(h)\circ \beta \ti\theta \beta^{-1}.$
The same result holds when $V$ is replaced by $F$ or $T$.
\end{proposition}

\section{Groupoid approach and generalisation of the main result}\label{sec:groupoidap}

In this section we adopt a groupoid approach.
We include all necessary definitions and constructions that are small modifications of the group case previously explained in the preliminary section. 
This leads to proofs of Theorem \ref{THB} and Corollary \ref{COR}. 

\subsection{Universal groupoids}\label{sec:univgroupoid}

We refer to \cite{GabrielZisman67} for the general theory on groupoids and groups of fractions. 

\begin{definition}
A small category $\cC$ admits a calculus of left-fractions if:
\begin{itemize}
\item (left-Ore's condition) For any pair of morphisms $p,q$ with same source there exists some morphisms $r,s$ satisfying $rp=sq$;
\item (Weak right-cancellative) If $pf=qf$, then there exists $g$ such that $gp=gq$.
\end{itemize}
\end{definition}

To any category $\cC$ can be associated a universal (or sometime called enveloping) groupoid $(\cG_\cC,P)$ together with a functor $P:\cC\to\cG_\cC.$
The groupoid $\cG_\cC$ has the same collection of objects than $\cC$ and morphisms are signed paths inside the category $\cC$: compositions of morphisms of $\cC$ and their formal inverse.
The next proposition shows that if $\cC$ admits a calculus of left-fractions then any morphism of $\cG_\cC$ can be written as $P(t)^{-1}P(s)$ for some morphisms of $t,s$ of $\cC$ with same target and thus justifies the terminology. 
The proof can be found in \cite[Chapter I.2]{GabrielZisman67}.

\begin{proposition}\label{prop:groupoid}
If $\cC$ admits a calculus of left-fractions, then any morphism of $\cG_\cC$ can be written as $P(t)^{-1}P(s)$ for $t,s$ morphisms of $\cC$ (having common target).
Using the fraction notation $\dfrac{t}{s}:=P(t)^{-1}P(s)$ we obtain that $\dfrac{ft}{fs}=\dfrac{t}{s}$ for any morphism $f$ of $\cC$.
Moreover, we have the following identities:
\begin{align*}
&\dfrac{t}{s}\dfrac{t'}{s'} =\dfrac{ft}{f's'} \text{ for any $f,f'$ satisfying } fs = f't' ; 
\text{ and } \left( \dfrac{t}{s} \right)^{-1} = \dfrac{s}{t} . 
\end{align*}
We say that $\cG_\cC$ is the groupoid of fractions of $\cC.$
\end{proposition}
\begin{remark}
A perfect analogy to Ore's work on embedding a semi-group into a group would be to have that the functor $P:\cC\to\cG_\cC$ is faithful and that morphisms of $\cG_\cC$ can be expressed as formal fractions of morphisms of $\cC$. 
This happens exactly when $\cC$ is cancellative and satisfies left-Ore's condition, see \cite[Proposition 3.1.1]{DDGKM15}.
However, for our study we do not need to have a faithful functor to the universal groupoid and only demand a calculus of left-fractions.
\end{remark}

\begin{remark}
If we fix an object $e$ of $\cC$, then the group of fractions $G_\cC$ associated to $(\cC,e)$ is the automorphism group $\cG_\cC(e,e)$ inside the universal groupoid $\cG_\cC.$
\end{remark}

\subsection{Jones' actions of groupoids}\label{sec:gpoidaction}

Consider a small category $\cC$ with a calculus of left-fractions and a functor $\Phi:\cC\to\cD$.
For \text{any} morphism $f$ of $\cC$ we consider the space $X_f$ that is a copy of $\Phi(\target(f)).$
We equipped the set of morphisms of $\cC$ with the order $f\leq f'$ if there exists $p$ such that $pf=f'$. Note that elements are comparable if and only if they have same source.
For any object $a\in\ob(\cC)$ we obtain a directed system $(X_f, \source(f)=a)$ with limit space $\scrX_a.$
Let $\tilde\scrX:= \oplus_{a\in\ob(\cC)} \scrX_a$ be their direct sum (inside the category of sets that is a disjoint union).
The set $\tilde\scrX$ can be described by equivalence classes of pairs $(f,x)$ with $f\in \cC(a,b),x\in\Phi(b)$ and $a\in\ob(\cC)$ where the equivalence relation is generated by $(f,x)\sim (hf,\Phi(h)x).$
Write $\dfrac{f}{x}$ for such a class that we call a fraction and observe that $\scrX_a$ corresponds to the fractions $\dfrac{f}{x}$ for which $\source(f)=a.$
Consider an element of the universal groupoid $\cG_\cC$ that we can write as a fraction of morphisms $\dfrac{f}{f'}$.
If $\dfrac{h}{x}$ is in $\scrX_a$ and that $\source(f')=a$, then we define the composition:
$$\dfrac{f}{f'}\cdot \dfrac{h}{x} = \dfrac{pf}{\Phi(q)x} \text{ where } pf'=qh.$$
Hence, any fraction $\dfrac{f}{f'}\in\cG_\cC$ defines a map from $\scrX_{\source(f')}$ to $\scrX_{\source(f)}.$
We define 
$$\pi\left( \dfrac{f}{f'} \right) \dfrac{h}{x}  = \dfrac{pf}{\Phi(q)x}$$
 and say that $(\pi,\tilde\scrX)$ is the Jones action of the groupoid $\cG_\cC$ on $\tilde\scrX.$

An example of particular interest for us is when $\cD$ is the category of Hilbert spaces $\Hilb.$
Given a functor $\Phi:\cC\to\Hilb$ we build a Hilbert space $$\tilde\scrH=\oplus_{a\in\ob(\cC)}\scrH_a$$ that is the direct sum of Hilbert spaces $\scrH_a$ which are the completion of 
$$\{ (f,\xi): f\in\cC(a,b), \xi\in\Phi(b), b\in\ob(\cC)\}/\sim$$ for objects $a\in\ob(\cC).$ 
We equip $\tilde\scrH$ with the inner product 
$$\langle \xi,\eta\rangle=\sum_{a\in\ob(\cC)} \langle \xi_a,\eta_a\rangle$$ where $\xi_a,\eta_a$ are the components of $\xi,\eta$ in $\scrH_a.$
Given a fraction $\dfrac{f}{f'}$ with $f\in\cC(a,b),f'\in\cC(a',b')$ we define a partial isometry $\pi\left( \dfrac{f}{f'} \right)$ on $\tilde\scrH$ with domain $\scrH_{a'}$ and range $\scrH_{a}$ satisfying $\pi\left( \dfrac{f}{f'} \right)\dfrac{f'}{\xi}=\dfrac{f}{\xi}.$
We say that $(\pi,\tilde\scrL)$ is a representation of the groupoid $\cG_\cC.$

\subsection{Important examples}\label{sec:examples}

{\bf Higman-Thompson's groups.}
If we consider $\cSF_k$ the category of $k$-ary symmetric forests, then it is a category that admits a calculus of left-fractions for $k\geq 2$. 
Note that $\cSF_2=\cSF$ is the category of binary symmetric forests which we worked with all along this article.
Observe that the group of automorphisms $\cG_{\cSF_k}(r,r)$ can be represented by pairs of symmetric $k$-ary forests with both $r\geq 1$ roots and the same number of leaves.
This is one classical description given in the article of Brown of the so-called Higman-Thompson's group $V_{k,r}$ \cite{Higman74,Brown87}.
Hence, the groupoid $\cG_{\cSF_k}$ contains (in the sense of morphisms) every Higman-Thompson's group $V_{k,r}$ for a fixed $k\geq 2$.

{\bf Larger categories.}
We consider larger categories made of symmetric forests and groups.
Fix $k\geq 2$ and consider a group $\Ga$ together with a morphism $\theta:\Ga\to\Ga.$ Define 
the morphism $S_k:\Ga\to \Ga^k, g\mapsto (\theta(g),e,\cdots,e)$.
We can now proceed as in Section \ref{sec:largercat} for constructing a monoidal functor $\Theta:\cSF_k\to\Gr$ and a larger category $\cC(k,\theta,\Ga)$.
The only difference being that morphisms of $\cSF_k$ are all composition of tensor products of the trivial tree $I$ and the unique \textit{$k$-ary} tree $Y_k$ (instead of the \textit{binary} tree $Y$) that has $k$ leaves. 
We then set $\Theta(1)=\Ga,\Theta(Y_k)=S_k$ and the definition of the larger category $\cC(k,\theta,\Ga)$ becomes obvious.
It is a category that admits a calculus of left-fractions.
By adapting Proposition \ref{prop:Bernoulli} we obtain the following:

\begin{proposition}\label{prop:HT}
Consider $k\geq 2$ and the identity automorphism $\theta=\id$.
Let $\cC_k$ be the category $\cC(k,\id,\Ga)$ and put $\cG_k$ its universal groupoid.
If $r\geq 1$, then the automorphism group $\cG_k(r,r)$ of the object $r$ is isomorphic to the wreath product 
$$\Ga\wr_{\Q_r(0,r)}V_{k,r}:=\oplus_{ \Q_k(0,r) }\Ga\rtimes V_{k,r}$$ 
for the classical action of the Higman-Thompson's group $V_{k,r}$ on the set $\Q_k(0,r)$ of $k$-adic rationals in $[0,r).$

More generally, if $\theta$ is any automorphism of $\Ga$, then $\cG_k(r,r)$ is isomorphic to the twisted wreath product 
$$\Ga\wr_{\Q_r(0,r)}^\theta V_{k,r}:=\oplus_{ \Q_k(0,r) }\Ga\rtimes^\theta V_{k,r}$$ 
where the action $V_{k,r}\act \oplus_{ \Q_k(0,r) }\Ga$ is the following:
$$(v\cdot a)(x):= \theta^{\log_k(v'(v^{-1}x))}(a(v^{-1}x)) \text{ for } v\in V_{k,r}, a\in \oplus_{ \Q_k(0,r) }\Ga, x\in \Q_k(0,r).$$
\end{proposition}

\begin{remark}
Note that given a fixed $k\geq 2$, we have that two objects $r_1,r_2$ of the universal groupoid $\cG_{\cSF_k}$ are in the same connected component if and only if $r_1=r_2$ modulo $k-1$.
In that case the automorphism groups of the objects $r_1$ and $r_2$ inside $\cG_{\cSF_k}$ are isomorphic (to see this: simply conjugate the first automorphism group by any morphism $f\in\cG_{\cSF_k}(r_1,r_2)$) and thus $V_{k,r_1}\simeq V_{k,r_2}$. 

The same argument applies to the wreath products associated to $\cC_k:=\cC(k,\theta,\Ga)$. This provides isomorphisms between various wreath products of the form $\Ga\wr_{ \Q_k(0,r) }^\theta V_{k,r}$.
In particular, if $k=2$, then all Higman-Thompson's groups $V_{2,r}$ (and wreath products $\Ga\wr^\theta_{ \Q_2(0,r) }V_{2,r}$ for fixed $(\Ga,\theta)$) are mutually isomorphic but this is no longer the case when $k$ is strictly larger than $2$.
\end{remark}

\subsection{Haagerup property for groupoids}\label{sec:gpoidH}

Haagerup property was defined for measured discrete groupoids by Anantharaman-Delaroche in \cite{Delaroche12}. 
Her work generalises two important cases that are countable discrete groups and measured discrete equivalence relations.
Our case is slightly different as fibers might not be countable. However, since the set of objects is countable we can study our groupoid in a similar way than a discrete group and avoid any measure theoretical considerations.

Let $\cG$ be a small groupoid with countably many objects. 
We recall what are representations and coefficients for $\cG$.
Identify $\cG$ with the collection of all morphisms of $\cG.$
A representation $(\pi,\scrL)$ of $\cG$ is a Hilbert space $\scrL$ equal to a direct sum $\oplus_{a\in\ob(\cG)} \scrL_a$ and a map $\pi:\cG\to B(\scrL)$ such that $\pi(g)$ is a partial isometry with domain $\scrL_{\source(g)}$ and range $\scrL_{\target(g)}$.
A coefficient of $\cG$ is a map $\phi:\cG\to\C, g\mapsto \langle \eta,\pi(g)\xi\rangle$ for a representation $(\pi,\scrL)$ and some unit vectors $\xi,\eta\in\scrL$.
The coefficient is positive definite (or is called a positive definite function) if $\eta=\xi.$
Note that equivalent characterizations of positive definite functions exist in this context but we will not need them. 
We define the Haagerup property as follows.
\begin{definition}
A small groupoid $\cG$ with countably many objects has the Haagerup property if there exists a net of positive definite functions on $\cG$ that converges pointwise to one and vanish at infinity.
\end{definition}
Assume that $\cG$ has countable fibers and is as above. Let $\mu$ be any strictly positive probability measure on the set of objects of $\cG$. Then we can equip $(\cG,\mu)$ with a structure of a discrete measured groupoids, see \cite{Delaroche12}.
The two notions of coefficients and positive definite functions coincide for $\cG$ and $(\cG,\mu)$.
Moreover, $\cG$ has the Haagerup property in our sense if and only if $(\cG,\mu)$ does in the sense of Anantharaman-Delaroche \cite{Delaroche12} which justifies our definitions.
The following property is obvious.
\begin{proposition}\label{prop:gpingpoid}
Let $\cG$ be a small groupoid with countably many objects. Consider a subgroupoid $\cG_0$ in the sense that $\ob(\cG_0)\subset \ob(\cG)$ and $\cG_0(a,b)\subset\cG(a,b)$ for any objects $a,b$ of $\cG_0.$
If $\cG$ has the Haagerup property, then so does $\cG_0$ and in particular every group $\cG(a,a)$ (considered as a discrete group) for $a\in\ob(\cG)$.
\end{proposition}

\begin{proof}[Proof of Theorem \ref{THB} and Corollary \ref{COR}]
Consider a discrete group $\Ga$ with the Haagerup property and an injective morphism $\theta:\Ga\to\Ga.$
This defines a map $S_k:\Ga\to\Ga^k$, a category $\cC=\cC(k,\theta,\Ga)$ with universal groupoid $\cG_\cC$ as explained above.
Note that $\cG_\cC$ is a small category with set of object $\N^*$ that is countable.
Let us prove that $\cG_\cC$ has the Haagerup property.

We prove the case $k=2$. The general case can be proved in a similar way.

{\bf Claim: We can assume that $\theta$ is an automorphism.}

This follows from \cite[Section 4.1]{Brothier22}.
Indeed, from $(\Ga,\theta)$ we construct a directed system of groups indexed by the natural numbers where all groups are $\Ga$ and the connecting maps are $\theta$.
The limit is a group $\widehat\Ga$ that admits an automorphism $\widehat\theta$.
Now, if $\Ga$ has the Haagerup property, then so does $\widehat\Ga$ since it is the limit of a group with the Haagerup property. 
Note, this fact uses crucially that $\theta$ is injective (and thus no quotients are performed).
Moreover, we prove in Proposition 4.3 of \cite{Brothier22} that the groupoid of fractions $\cG_{\cC}$ of $\cC(2,\theta,\Ga)$ is isomorphic to the groupoid of fractions of the category $\cC(2,\widehat\theta,\widehat\Ga)$.

From now one we assume that $\theta$ is an automorphism.
Consider a pair $(\rho,R)$ constructed from a positive definite coefficient $\phi_\Ga:\Ga\to\C$ vanishing at infinity and an isometry $R_\al$ for some $0<\al<1$ as in Section \ref{sec:coef}.
Assume that there exists $0\leq c<1$ such that $|\phi_\Ga(g)|<c$ for any $g\neq e.$
This defines a functor $\Psi:\cC\to\Hilb$ that provides a representation $(\pi,\tilde\scrL)$ of the universal groupoid $\cG_\cC$ satisfying that 
$$\pi\left( \dfrac{g\sigma f}{ f' }\right) \dfrac{f'}{\xi} = \dfrac{pf}{ \Tens(\sigma^{-1})\rho^{\ot n}(g^{-1}) \Psi(q) \xi}$$
for $f,f'$ forests with $n$ leaves, $\sigma\in S_n$ and $g\in\Ga^n$.
Consider the unit vector 
$$\eta_{N,\phi_\Ga,\al}:=N^{-1/2} \oplus_{n=1}^N \xi\ot\de_e$$ for $N\geq 1$ and where $\xi$ is the vector satisfying $\phi_\Ga(g)=\langle \xi,\kappa_0(g)\xi\rangle$, see Section \ref{sec:coef}.
By following the same proof than Proposition \ref{prop:Czero} we obtain that the coefficient $\varphi_{N,\phi_\Ga,\al}$ associated to $\eta_{N,\phi_\Ga,\al}$ and $(\pi,\tilde\scrL)$ vanishes at infinity.
Fix a net of positive definite functions $(\phi_i:\Ga\to\C, i\in I)$ satisfying the hypothesis of the Haagerup property such that $|\phi_i(g)|<c_i$ for any $g\neq e$ for some $0\leq c_i <1$.
The net of coefficients 
$$(\varphi_{N,\phi_i,\al}, N\geq 1, i\in I, 0<\al<1)$$ 
on the groupoid $\cG_\cC$ satisfies all the hypothesis required by the Haagerup property.
This proves Theorem \ref{THB}.

Consider the category $\cC=\cC(2,\theta,\Ga)$ where $\Ga$ has the Haagerup property, $\theta\in\Aut(\Ga)$ is an automorphism, and the category of $k$-ary forests $\cSF_k$.
By Proposition \ref{prop:gpingpoid} we have that the group $\cG_\cC(r,r)$ of automorphisms of the object $r$ in the universal groupoid of $\cC$ is isomorphic to the twisted wreath product $\Ga\wr^\theta_{\Q_k(0,r)} V_{k,r}$.
We proved that $\cG_\cC$ has the Haagerup property and thus so does the isotropy group $\cG_\cC(r,r)$ (by Proposition \ref{prop:gpingpoid}). This proves Corollary \ref{COR}.
\end{proof}

\appendix
\section{Categories and groups of fractions}

We end this article by providing an alternative description of Jones' actions using a more categorical language. We do not give details and only sketch the main steps.
This was explained to us by Sergei Ivanov, Richard Garner and Steve Lack. We are very grateful to them.

We keep the notation of Section \ref{sec:actions} and thus $\Phi:\cC\to\cD$ provides a Jones' action $\pi_\Phi:G_\cC\act \mathscr X$ with $\mathscr X=\varinjlim_{t,\Phi} X_t.$
Let $(\cG_\cC,P)$ be the universal groupoid of $\cC$ with functor $P:\cC\to\cG_\cC$.
Let $(e\downarrow \cC)$ be the comma-category of objects under $e$ whose objects are morphisms of $\cC$ with source $e$ and morphism triangles of morphisms of $\cC$ (e.g.~if $\cC=\cF,e=1$, then objects and morphisms of $(1\downarrow \cF)$ are trees and forests respectively). 
This category comes with a functor $(e\downarrow \cC)\to \cC$ consisting in only remembering the target of morphisms (e.g.~sending a tree to its number of leaves and keeping forests for morphisms). 
The composition of functors $\tilde\Phi:(e\downarrow \cC)\to\cC\to\cD$ provides a diagram of type $(e\downarrow\cC)$ in the category $\cD$ and the colimit (if it exists) corresponds to our limit $\mathscr X.$
Assume that the left Kan extension $Lan_P(\Phi):\cG_\cC\to\cD$ of $\Phi$ along $P$ exists. 
Then one can prove that $Lan_P(\Phi)(e)$ is isomorphic to the colimit of $\tilde\Phi$ and is thus isomorphic to $\mathscr X.$ 
But then $Lan_P(\Phi)$ sends $\cG_\cC(e,e)\simeq G_\cC$ in the automorphism group of $\mathscr X$ which corresponds to the Jones' action $\pi_\Phi$.

Using this construction, if we only want a map from the group of fractions $G_\cC$ to the automorphism group of an object, then we don't need to require that objects of $\cD$ are sets.
Actions of the whole universal groupoid $\cG_\cC$ can be constructed in a similar way.
In order to make this machinery working we need to have a target category $\cD$ with sufficiently many colimits in order to have a Kan extension of our functor. 



\end{document}